\documentclass[11pt]{amsart}
\usepackage{tikz-cd}
\usepackage{amsmath}
\usepackage{amssymb}
\usepackage{mathrsfs}
\usepackage{float}
\usepackage[all]{xy}
\usepackage{quiver}

\usepackage{comment}
\newtheorem{theorem}{Theorem}[section]
\newtheorem{claim}[theorem]{Claim}

\newtheorem{lemma}[theorem]{Lemma}

\newtheorem{proposition}[theorem]{Proposition}
\newtheorem{corollary}[theorem]{Corollary}

\theoremstyle{definition}
\newtheorem{definition}[theorem]{Definition}
\newtheorem{example}[theorem]{Example}

\newtheorem{question}[theorem]{Question}

\theoremstyle{remark}
\newtheorem{remark}[theorem]{Remark}

\newtheorem{notation}[theorem]{Notation}

\def\l{{\langle}}
\def\r{{\rangle}}

\newcount\skewfactor
\def\mathunderaccent#1#2 {\let\theaccent#1\skewfactor#2
\mathpalette\putaccentunder}
\def\putaccentunder#1#2{\oalign{$#1#2$\crcr\hidewidth
\vbox to.2ex{\hbox{$#1\skew\skewfactor\theaccent{}$}\vss}\hidewidth}}

\newcommand{\lusim}[1]{\smash{\underset{\raisebox{1.2pt}[0cm][0cm]{$\sim$}}
{{#1}}}}
\def\smallbox#1{\leavevmode\thinspace\hbox{\vrule\vtop{\vbox
   {\hrule\kern1pt\hbox{\vphantom{\tt/}\thinspace{\tt#1}\thinspace}}
   \kern1pt\hrule}\vrule}\thinspace}

\DeclareMathOperator{\dom}{dom}

\DeclareMathOperator{\crit}{crit}

\DeclareMathOperator{\rng}{rng}
\DeclareMathOperator{\hull}{Hull}

\title{Applications of the Magidor iteration to ultrafilter theory}
\author{Tom Benhamou}
\thanks{The first author's research  was supported by the National Science Foundation under Grant
No. DMS-2346680.}
\address[Benhamou]{Department of Mathematics, Rutgers University, New Brunswick, NJ USA}
\email{tom.benhamou@rutgers.edu}
\author{Gabriel Goldberg}
\thanks{The second author's research was supported by the National Science Foundation under Grant No. DMS-2401789.}
\address[Goldberg]{Departement of Mathematics, UC Berkeley, Berkeley, CA 94720-3840 USA }
\email{ggoldberg@berkeley.edu}

\begin{document}
\begin{abstract}
    We characterize sums of normal ultrafilters after the Magidor iteration of Prikry forcings over a discrete set of measurable cardinals. We apply this to show that the weak Ultrapower Axiom is not equivalent to the Ultrapower Axiom. We also construct a non-rigid ultrapower and two uniform ultrafilters on different cardinals that have the same ultrapower. 
\end{abstract}
\maketitle
\section{Introduction}
Iterated Prikry forcing was first introduced by Magidor \cite{MAGIDOR197633} in his seminal \textit{``study on identity crises"} to produce a model of ZFC in which the least measurable cardinal is strongly compact.
The rough idea is to iteratively singularize each cardinal \(\alpha\) in some set of measurable cardinals \(\Delta\) using the Prikry forcing associated with a normal measure on \(\alpha\).

Ben-Neria \cite{Ben-Neria2014-BENFMI} was the first to notice that if Magidor's construction is carried out over the core model \(K\), then the normal ultrafilters of the forcing extension can be classified in terms of the normal ultrafilters of \(K\). %(More precisely, Ben-Neria showed that if we force with the Magidor iteration over the core model below $o^{\P}$, then every normal ultrafilter in the generic extension can be classified in terms of the core model normal ultrafilters on measurables that are not destroyed by the Magidor iteration.)
Recently, Ben-Neria's work was substantially extended and generalized by Kaplan \cite{EyalMagidor}, who showed, most significantly, that the classification could be carried out even when the ground model is not the core model.
%\begin{theorem}[Kaplan]
    %Assume GCH. Let $\kappa$ be a measurable cardinal, let $\mathbb{P}$ be the Magidor iteration of Prikry forcings below $\kappa$, and let $G\subseteq \mathbb{P}$ be a generic filter. Then there is a natural one-to-one correspondence between normal ultrafilters  on $\kappa$ of \(V[G]\) and normal ultrafilters on $\kappa$ of \(V\).
    %Moreover, if \(W\) is a normal ultrafilter on $\kappa$ of \(V[G]\), then $j_W\restriction V$ is an iterated ultrapower by normal ultrafilters of $V$. 
%\end{theorem}

This paper is focused on the special case of Magidor's construction in which the set of measurable cardinals $\Delta$ to be singularized is \textit{discrete} in the sense that it does not contain any of its limit points. In this restricted setting, the forcing can be viewed as a product of Prikry forcings which we will refer to as the \textit{Magidor product}. 
%This forcing was considered in \cite{Fuchs2005-FUCACO} by Fuchs which provided a Mathias criterion for genericity. 
For discrete Magidor products, Kaplan's theorem can be stated as follows:
\begin{theorem}[Kaplan]
    \label{cor: cor from Kaplan}
    Suppose \(G\) is \(V\)-generic for a Magidor product of Prikry forcings on a discrete set \(\Delta\subseteq \kappa\). Then every normal measure $U$ on $\kappa$ in $V$ generates a normal ultrafilter $U^*$ in \(V[G]\). Moreover $j_{U^*}\restriction V =i\circ j_U$, where $i$ is an iterated ultrapower of $V_U$.
\end{theorem}
Here and below, \[j_U :V \to V_U\] denotes the ultrapower of the universe of sets by $U$. 
%It also follows from \cite[Cor. 3.19]{EyalMagidor}, that in the discrete context, $i$ is the so-called \textit{complete iteration} (see Definition \ref{def: complete iteration}). 

This paper addresses the question of
whether Theorem \ref{cor: cor from Kaplan}
can be generalized to ultrafilters \(U\) that are not normal.
Although we do prove a slight generalization of Kaplan's result to a collection of measures we call \textit{mild} (see \ref{Cor: Kaplan's Theorem}),
the more surprising contribution of this paper is a negative answer to the general question: we show in Section \ref{Section: Some ultrafilters} that in the context of Theorem \ref{cor: cor from Kaplan}, a $\kappa$-complete ultrafilter \(U\) on \(\kappa\)
can extend in unexpected ways after a Magidor product of Prikry forcings.
In particular, the filter generated by \(U\) in the forcing extension is not an ultrafilter. Moreover, these unexpected extensions answer several natural questions in the theory of ultrafilters.

For example, recall that Kunen's inconsistency theorem
states that there is no nontrivial elementary embedding from the universe of sets to itself.
Suppose \(U\) is a countably complete
ultrafilter, and let \(V_U\) be the ultrapower of the universe of sets by \(U\). Can there be a nontrivial elementary embedding from \(V_U\) to itself? 
By \cite[Theorem 3.3, Theorem 4.29]{GoldbergUniqueness}, the answer is no if \(V = \text{HOD}\) or
if \(U\) is \(\kappa^+\)-complete where \(\kappa\) is extendible. Nevertheless, we answer the question positively assuming the consistency of a measurable limit of measurable cardinals:

\begin{theorem}
    It is consistent with \textnormal{ZFC} that for some normal ultrafilter \(U\), there is an elementary embedding from \(V_U\) to itself.
\end{theorem}

We then consider the question of ultrafilters
with the same ultrapower. 
Suppose \(U\) and \(W\)
are countably complete ultrafilters 
such that \(V_U = V_W\).
Must \(U\) and \(W\) be isomorphic (that is, Rudin--Keisler equivalent)? Woodin observed that the answer is no \cite[Theorem 3.1]{GoldbergUniqueness}, though again this is true if \(V = \text{HOD}\). We consider the weaker question: must there exist \(X\in U\) and \(Y\in W\) with
\(|X| = |Y|\)? Again, we show the answer is no after a discrete Magidor product of Prikry forcings, assuming the consistency of a limit of measurable cardinals of measurable cofinality:

\begin{theorem}
    It is consistent with \textnormal{ZFC} that there are countably complete uniform ultrafilters \(U\) and \(W\) on distinct cardinals such that \(V_U = V_W\).
\end{theorem}

Recall that the Ultrapower Axiom (UA) \cite{GoldbergUA} states that for any two countably complete ultrafilters $U,W$, there are $W'\in V_U$ and $U'\in V_{W}$ such that \(W'\) is a countably complete ultrafilter in \(V_U\),
\(U'\) is a countably complete ultrafilter in \(V_W\), $(V_U)_{W'}=(V_W)_{U'}$ and $j_{W'}\circ j_U=j_{U'}\circ j_W$. The Weak Ultrapower Axiom (Weak UA) is the same statement, omitting the requirement that $j_{W'}\circ j_U=j_{U'}\circ j_W$.

The conclusions of the previous theorems are incompatible with UA by \cite[Theorem 5.2]{GoldbergUniqueness}. Therefore UA typically becomes false after performing a discrete Magidor product of Prikry forcings.
Nevertheless, starting from a model of UA containing no measurable cardinals \(\delta\) of Mitchell order \(2^{2^\delta}\), our main technical result (Theorem \ref{Thm: Classification}) yields a 
classification of the lifts of \(\kappa\)-complete ultrafilters on \(\kappa\) to the forcing extension.

This classification suffices to show in certain cases that Weak UA holds in the forcing extension.
This answers a question of the second author \cite[Question 9.2.4]{GoldbergUA}, assuming there is a measurable limit of measurable cardinals:
\begin{theorem}
    The Weak Ultrapower Axiom does not imply the Ultrapower Axiom. 
\end{theorem}

The above results follow from our main technical theorem which is an analysis of all possible extensions of $\kappa$-complete ultrafilters over $\kappa$ whose ultrapower can be factored into a finite iterated ultrapower by normal ultrafilters.
The following consequence gives a sense of what this classification entails:
\begin{theorem}
    Suppose that $W\in V$ is a $\kappa$-complete ultrafilter over $\kappa$ that can be factored into a finite iteration of normal ultrafilters. Let $G$ be $V$-generic for a discrete Magidor product of Prikry forcings. Then in $V[G]$, $W$ extends to at most countably many countably complete ultrafilters.
    
    Moreover, if $W^*$ is an extension of $W$, then $j_{W^*}\restriction V=i\circ e\circ j_W$, where $e:V_W\rightarrow N$ is a finite external iteration of $V_W$ by normal measures and $i$ is an internal iterated ultrapower of $N$ by normal measures.
\end{theorem}
The previous theorem also gives a sense in which our classification generalizes Kaplan's result.
%The outcome of the classification of the $V[G]$ ultrapower embeddings is a description of their  restrictions as iterated ultrapowers by normal measures of the ground model, similar to the one proven by Kaplan:
%\begin{theorem}
    %In the notations of the previous theorem,   
%\end{theorem}

This paper is organized as follows:
\begin{itemize}
    \item In Section~\S\ref{Section: prel} we prove some preliminary results regarding ultrapowers and the discrete Magidor product.
    \item In Section~\S\ref{Section: Complete iteration} we study the complete iteration by a sequence of normal measures and how it relates to the restriction of an ultrapower of the generic extension to the ground model.
    \item In Section~\S\ref{Section: Some ultrafilters} we describe an ultrafilter which has many extensions to the generic extension of the discrete Magidor product and characterize its extensions. Also, we provide the example of a non-rigid ultrapower.
    \item In Section~\S\ref{Section: Classification} we prove our characterization of lifts of ultrafilters to the generic extension by the discrete Magidor product.
    \item In Section~\S\ref{Section:Applications} we prove some applications of our characterization: a model of weak UA which fails to satisfy UA, and an example of the same ultrapower by ultrafilters on different cardinals.
    \item In Section~\S\ref{Section: Problems} we present some related open problems. 
\end{itemize}
\section{Preliminaries}\label{Section: prel}
\subsection{Derived ultrafilters and commuting squares}
Let $P$ be a transitive model of set theory and let $U$ be a (possibly external)  $P$-ultrafilter over a set $X\in P$.\footnote{That is, $U$ is an ultrafilter on the Boolean algebra $P(X)\cap P$.} We will always assume that $U$ is countably complete; namely, for any countable  $\mathcal{A}\subseteq U$, $\bigcap\mathcal{A}\neq\emptyset$. Denote by $j^P_U:P\to P_U$ the ultrapower of $P$ by $U$ using functions $f:X\to P$ in $P$. We will suppress the superscript $P$ from $j^P_U$ whenever there is no ambiguity. Since $U$ is countably complete, $P_U$ is well-founded and we identify $P_U$ with its transitive collapse.

Let $j:P\to Q$ be an elementary embedding of transitive models of set theory, $X\in P$, and $a\in j(X)$. The \textit{$P$-ultrafilter on $X$ derived from $j$ and $a$} is the set
$$U=\{A\in P(X)\cap P\mid a\in j(A)\}$$
(The underlying set $X$ is typically suppressed.)
It is well known that $U$ is a $P$-ultrafilter and that the map $k:P_U\to Q$ defined by $k([f]_U)=j(f)(a)$ is the unique elementary embedding mapping $[id]_U$ to $a$ such that $k\circ j_U=j$. The following generalization of this fact will be used implicitly in many calculations below:

\begin{lemma}[Shift lemma]\label{lemma: the shift lemma}
Suppose $i:P\rightarrow Q$ is an elementary embedding, $W$ is a $Q$-ultrafilter over a set in $\rng(i)$ and $U=i^{-1}[W]$. Then the embedding $k:P_U\rightarrow Q_W$ defined by $k([f]_U)=[i(f)]_W$ is  well-defined and elementary. Moreover, it is the unique embedding mapping $[id]_U$ to $[id]_W$ such that the following diagram commutes:
% https://q.uiver.app/#q=WzAsNCxbMCwxLCJNIl0sWzAsMCwiTiJdLFsxLDAsIk5fVyJdLFsxLDEsIk1fVSJdLFswLDEsImkiXSxbMSwyLCJqX1ciXSxbMCwzLCJqX3tVfSIsMl0sWzMsMiwiayIsMl1d
\[\begin{tikzcd}
	Q & {Q_W} \\
	P & {P_U}
	\arrow["{j_W}", from=1-1, to=1-2]
	\arrow["i", from=2-1, to=1-1]
	\arrow["{j_{U}}"', from=2-1, to=2-2]
	\arrow["k"', from=2-2, to=1-2]
\end{tikzcd}\]
\end{lemma}
\begin{proof}
    Note that $U$ is the $P$-ultrafilter derived from $j_W\circ i$ and $[id]_W$. Moreover, $k$ is the associated factor map, and the lemma follows. 
\end{proof}
Suppose \(i : M\to N\) is an elementary embedding and \(X\subseteq N\).
Then \(\hull^N(i[M]\cup X)\) denotes
\(\{i(f)(s) : s\in X^{<\omega}, f\in M\}\).
This is the smallest elementary substructure of \(N\) containing \(i[M]\cup X\).
\begin{lemma}
    Suppose that $D$ is a $P$-ultrafilter, $j:P\rightarrow M$, $j':P_D\rightarrow M'$, and $k:M\rightarrow M'$ is such that $M'=\hull^{M'}(\rng(k)\cup\rng(j'))$ and the following diagram commutes:
    \[\begin{tikzcd}
	P_D & {M'} \\
	P & {M}
	\arrow["{j'}", from=1-1, to=1-2]
	\arrow["j_D", from=2-1, to=1-1]
	\arrow["{j}"', from=2-1, to=2-2]
	\arrow["k"', from=2-2, to=1-2]
\end{tikzcd}\]
    Then if $D'$ is the $M$-ultrafilter derived from $k$ and $j'([id]_D)$,
    $M_{D'}=M'$, $k=j_{D'}$, and \([id]_{D'} = j'([id]_D)\).
\end{lemma}
\begin{proof}
    Note that there is a factor map $k':M_{D'}\rightarrow M'$ defined by $k'([f]_{D'})=k(f)(j'([id]_D))$ such that $k'\circ j_{D'}=k$. To see that $k'$ is onto, we have \begin{align*}      
        M'&=\hull^{M'}(\rng(k)\cup\rng(j'))\\
        &=\hull^{M'}(\rng(k)\cup \rng(j'\circ j_D)\cup\{j'([id]_D)\})\\
    &=\hull^{M'}(\rng(k)\cup\{j'([id]_D)\})\subseteq \rng(k')\end{align*} 

    Since \(k'\) is an isomorphism of transitive models, \(k'\) is the identity. Also \([id]_{D'} = k'([id]_{D'}) = j'([id]_D)\), proving the lemma.
\end{proof}
\begin{corollary}\label{cor: hulls and ultrapowers}
    If $k:M\rightarrow M'$, $M'=\hull^{M'}(\rng(k)\cup\{a\})$, and $a\in k(X)$ for some $X\in M$, then $M_D=M'$ and $k=j_D$ where $D$ is the $M$-ultrafilter over $X$ derived from $a$ and $k$ and \([id]_D = a\). 
\end{corollary}
\begin{proof}
    This is a special case of the previous lemma where $j=id$ and $j'$ is the factor embedding.
\end{proof}

\subsection{
Iterations and products of Prikry-type forcings}

%{\color{blue} consider greater generality
%\begin{definition}
%     A Prikry type forcing is a tripel $(P,\leq,\leq^*)$ such that $\leq^*\subseteq \leq$... \end{definition}}
 We define the \textit{Magidor support iteration} $\l \mathbb{P}_\alpha,\lusim{Q}_\beta\mid \alpha\leq \kappa, \beta<\kappa\r$ associated with a sequence of names \(\l \lusim{U}_\alpha\mid \alpha\in \Sigma\r\) where $\Sigma\subseteq \kappa$ is a set of measurable cardinals.
\begin{definition}[\cite{MAGIDOR197633,Gitik2010}] The forcing $\mathbb{P}_\alpha$ is defined inductively for each $\alpha\leq\kappa$: a condition $p$ belongs to $\mathbb{P}_\alpha$ if it is a function $p$ with $\dom(p)=\alpha$ such that:
\begin{enumerate}
    \item for every $\beta<\alpha$, $p\restriction\beta\in\mathbb{P}_\beta$.
    \item for every $\beta<\alpha$, $p\restriction \beta\Vdash_{\mathbb{P}_\beta}p(\beta)\in \lusim{Q}_\beta$, where $\lusim{Q}_\beta$ is trivial unless $\beta\in \Sigma$, in which case $\Vdash_{\mathbb{P}_\beta}\lusim{U}_\beta$ is a normal ultrafilter over $\beta$ and $\lusim{Q}_\beta$ is a $\mathbb{P}_\beta$-name for the usual Prikry forcing\footnote{Recall that $Pr(U)$ consists of pairs $(t,A)\in [\kappa]^{<\omega}\times U$, where $\max(t)<\min(A)$. It is equiped with two orders, $(t,A)\leq (s,B)$ iff $t\cap\max(s)=s$ and $A\subseteq B$, and $(t,A){\leq^*}(s,B)$ iff $t=s$ and $A\subseteq B$. Thus $1_{Pr(U)}=(\emptyset,\kappa)$.} $Pr(\lusim{U}_\beta)$ from \cite{Prikry}.  
    \item  there is a finite set $b_p$ such that for every $\beta\in \alpha\setminus b_p$, $p\restriction\beta\Vdash_{\mathbb{P}_\beta} p(\beta)\leq^* 1_{\lusim{Q}_\beta}$. 
    
\end{enumerate} 
We define $p\leq q$ iff for every $\beta<\alpha$, $p\restriction \beta\Vdash_{\mathbb{P}_\beta}p(\beta)\leq q(\beta)$. Also define $p\leq^* q$ iff for every $\beta<\alpha$, $p\restriction \beta\Vdash_{\mathbb{P}_\beta}p(\beta)\leq^* q(\beta)$.
\end{definition}
Magidor \cite[\S 2]{MAGIDOR197633} originally defined his forcing so that for a condition $p\in\mathbb{P}_\alpha$, and $\beta\in \Sigma\cap \alpha$, $p(\beta)=(t,\lusim{A})$. where $t^p_\beta$ is a finite sequence of ordinals and $\lusim{A}$ is a $\mathbb{P}_\beta$-name such that $\Vdash_{\mathbb{P}_\beta}\lusim{A}\in \lusim{U}_\beta$. A straightforward inductive argument shows that this is a dense set of conditions in the forcing defined above.

A set of ordinals $\Delta$ is called \textit{discrete} if it contains none of its accumulation points; that is, for every $\alpha\in \Delta$, $\sup(\Delta\cap\alpha)<\alpha$. In particular, $\Delta$ is an extremely thin non-stationary set, and so are all its restrictions to $\alpha<\sup(\Delta)$; similarly, if $\Delta\subseteq \kappa$ is discrete, then for every $\kappa$-complete ultrafilter $U$ over $\kappa$, $\kappa\notin j_U(\Delta)$. Note that if $\kappa$ is the least measurable limit of measurable cardinals, then the set of measurables below $\kappa$ is discrete. 

\begin{notation}
    For the rest of this section, we fix a discrete set $\Delta\subseteq \kappa$ of measurable cardinals and a sequence $\vec{U}=\l U_\alpha\mid \alpha\in \Delta\r$ such that $U_\alpha$ is a normal measure of Mitchell order $0$ over $\alpha$.
\end{notation}  We will consider the Magidor iteration associated to $\l \overline{U}_\beta\mid \beta\in \Delta\r$ where $\overline{U}_\beta$ is a $\mathbb{P}_\beta$-name for the filter generated by $U_\beta$  in $V^{\mathbb{P}_\beta}$ (see Proposition \ref{Prop: Properties of Magidor iteration}(2)).
\begin{remark}
Since the set $\Delta$ is discrete, a Magidor support here is always going to be non-stationary. Hence the forcing above inherits the properties of both the non-stationary and Magidor support iterations of Prikry-type forcings on the set $\Delta$.   
\end{remark}

\begin{proposition}\label{Prop: Properties of Magidor iteration}
    \begin{enumerate}
    \item  $\mathbb{P}_\kappa$ is $\kappa^+$-cc.
    \item If $\alpha\in \Delta$, $U_\alpha$ generates a normal ultrafilter in $V^{\mathbb{P}_\alpha}$. 
    \end{enumerate}
\end{proposition}
Item $(2)$ follows from L\'evy-Solovay. For the proof of $(1)$ see \cite[4.4]{MAGIDOR197633}.
\begin{proposition}
The set of conditions $p\in \mathbb P_\kappa$ such that for every $\beta\in \Delta$, $p(\beta) = \check{q}$ for some $q\in Pr(U_\beta)$ is $\leq$-dense in $\mathbb{P}_\kappa$.
\end{proposition}

\begin{proof}
  Let $p\in \mathbb{P}_\kappa$ be a condition, and without loss of generality, assume that for each $\beta\in \Delta$, $p(\beta)=(t_\beta,\lusim{A}_\beta)$, where $t_\beta\in [\beta]^{<\omega}$. For $\beta\in\Delta$, define $p^*(\beta)=(t_\beta,A^*_\beta)$ by setting $A^*_\beta=\{\gamma<\beta\mid p\restriction\beta\Vdash_{\mathbb{P}_\beta}\check{\gamma}\in \lusim{A}_\beta\}$. A L\'evy-Solovay-like argument shows that $p^*\in \mathbb{P}_\kappa$, and this is clearly a direct extension of $p$.
\end{proof}
If \(p \in \mathbb P_\alpha\) belongs to the dense set of the previous proposition, for \(\beta\in \Delta\cap \alpha\), we denote by \((t^p_\beta,A^p_\beta)\) the unique $q\in Pr(U_\beta)$ such that $p(\beta) = \check{q}$.

%Since for each $\alpha\in \Delta$, 
%Again, as our support is in fact non-stationary, we have the following lemma sue to Ben-Neria and Unger (see \cite[Lemma 1.3]{GitikKaplanNonStationary} or \cite[Lemma 2.3]{BenNeriaUngerHomChanges}):

%{\color{blue}\begin{lemma}[The fusion lemma]
%Let $p\in \mathbb{P}_\kappa$ and $e:\kappa\rightarrow V$ be a function such that $e(\alpha)$ is a $\mathbb{P}_{\alpha+1}$-name such that $p\restriction \alpha+1\Vdash_{\mathbb{P}_{\alpha+1}} e(\alpha)$ is a $\leq^*$-dense subset of $\mathbb{P}_{[\alpha+1,\kappa)}$. Then there is $p^*\leq^* p$ and a club $C$ such that for every $\alpha\in C$, $p^*\restriction\alpha+1\Vdash p^*\restriction[\alpha+1,\kappa)\in e(\alpha)$.
    
%\end{lemma}
%\begin{corollary}
%    Let $U$ be a $\kappa$-complete ultrafilter over $\kappa$, and $\sigma$ a sentence in the forcing language over $j_U(\mathbb{P}_\kappa)$. Then there is $p\in G$ such that $j_U(p)||\sigma$.
%\end{corollary}
%\begin{proof}
%    Let $p$ be any condition, and let $\sigma_\alpha$ for $\alpha<\kappa$ be the reflection of $\sigma$ to $\mathbb{P}_\kappa$. Consider the set of the $\mathbb{P}_{\alpha+1}$-name $e(\alpha)$ of all conditions $q\in\mathbb{P}_{[\alpha+1,\kappa)}$ such that $q$  find $q\leq^* j_U(p)$ such that $q \Vdash \sigma_\alpha$ or $q\Vdash \neg \sigma_\alpha$. 
%\end{proof}}
\begin{definition}
    The \textit{Magidor support product}
    \(\prod_{\alpha \in \Delta}\textit{Pr}(U_\alpha)\)
    consists of all functions
    \(p\) such that \(\dom(p) = \Delta\),
    \(p(\alpha)\in \textit{Pr}(U_\alpha)\),
    and \(\{\alpha\in \Delta : p(\alpha) \leq^* 1\}\) is cofinite.
\end{definition}
By the previous proposition, we obtain the following corollary:
\begin{corollary}[Magidor]
     $\mathbb{P}_\kappa$ is forcing equivalent to the Magidor support product $\prod_{\beta\in \Delta\cap\alpha}Pr(U_\beta)$.
\end{corollary}
We shall often treat $\mathbb{P}_\kappa$ as if it were the product $\prod_{\beta\in \Delta}Pr(U_\beta)$.

As usual in iterated forcing, one can construct
factor forcings \(\mathbb P_{[\alpha, \beta)}\in V^{\mathbb P_\alpha}\) such that \(\mathbb P_\alpha * \mathbb P_{[\alpha, \beta)} \cong \mathbb P_\beta\). This works in general for the Magidor iteration (and other iterations of Pr\'ikry-type forcings), but in the discrete case considered here, the factor forcing takes a simpler form: it is forcing equivalent to the Magidor product
\(\prod_{\xi \in \Delta\cap [\alpha, \beta)} \textit{Pr}(U_\xi)\). (One can use compute this forcing in \(V\) or \(V^{\mathbb P_\alpha}\); the former factor forcing is dense in the latter.)
We also define \(\mathbb P_{(\alpha,\beta)} =\prod_{\xi \in \Delta\cap (\alpha, \beta)} \textit{Pr}(U_\xi) \), etc.

\begin{lemma}[Prikry Lemma] For every $p\in\mathbb{P}_\kappa$, and every sentence $\sigma$ in the forcing language for $\mathbb{P}_\kappa$, there is $p^*\leq^* p$ such that $p^*{}||{} \sigma$.

\end{lemma}
    \subsection{On ground model ultrafilters that generate ultrafilters in the extension}
    In this subsection, we provide a version of Theorem \ref{cor: cor from Kaplan} which provides slightly more information in the context of discrete sequences of measures.
\begin{definition}
    Suppose \(\lambda\) is a measurable cardinal and $\Sigma\subseteq \lambda$ is unbounded. 
    A $\lambda$-complete ultrafilter $U$ over $\lambda$ is called \textit{$\Sigma$-mild} if there is a function $f:\lambda\rightarrow \lambda$ such that
    $$ 
[id]_U\leq j_U(f)(\lambda)<\min(j_U(\Sigma)\setminus\lambda)$$
\end{definition}

Note that if $U$ is a normal measure with $\Sigma\notin U$, then $U$ is $\Sigma$-mild. Also, every $\Sigma$-mild ultrafilter is a $p$-point.

%    If $f_*(U)=U$, then there is $X\in U$ such that $f\restriction X$ is the identity.
%$\alpha\in j_U(X)$ iff %$[id]_U\in j_U(X)$
%\begin{lemma}
%    If $U$ is an ultrafilter and $G$ is generic such that for each name $\dot{A}$ of a subset of $\kappa$, there is $p\in G$ such that $j_U(p)|| [id]_U\in j_U(\dot{A})$, then $U$ generates an ultrafilter in $V[G]$.
%\end{lemma}
%\begin{proof}
%    Define $$F=\{(\dot{A})_G\mid \exists p\in G, \ j_U(p)\Vdash [id]_U\in j_U(\dot{A})\}$$
%    Clearly, $U\subseteq F$. Also for every $(\dot{A})_G\in F$, $j_U(p)\Vdash [id]_U\in j_U(\dot{A})$ and therefore $B=\{\alpha<\kappa\mid p\Vdash \alpha\in \dot{A}\}\in U$. But since $p\in G$, then $B\subseteq (\dot{A})_G$. Hence $U$ generates $F$. To see that $F$ is an ultrafilter, let $A\in V[G]$ be a subset of $\kappa$ and $\dot{A}$ be a name for it. By our assumption, there is $p\in G$ such that $j_U(p)||[id]_U\in j_U(\dot{A})$, which directly implies that either $A\in F$ or $\kappa\setminus A\in F$.
%\end{proof}
%\begin{definition}
%    Let $\Sigma\subseteq \kappa$ be unbounded. A $\lambda$-complete ultrafilter $U$ over $\lambda$ is called \textit{$\Sigma$-mild} if there is a function $f:\lambda\rightarrow \lambda$ such that
 %   $$ 
%[id]_U\leq j_U(f)(\lambda)<\min(j_U(\Sigma)\setminus\lambda)$$
%\end{definition}

%    If $f_*(U)=U$, then there is $X\in U$ such that $f\restriction X$ is the identity.
%$\alpha\in j_U(X)$ iff %$[id]_U\in j_U(X)$

\begin{theorem}\label{Thm: Delta-mild generates an ultraiflter}
If $U$ is a $(\Delta\cap \lambda)$-mild ultrafilter over a cardinal $\lambda\notin\Delta$, then $U$ generates a $\lambda$-complete ultrafilter in $V^{\mathbb P_\kappa}$.
\end{theorem}
\begin{proof}
    Let \(G\subseteq \mathbb P_\lambda\) be \(V\)-generic. Note that it suffices to show that \(U\) generates an ultrafilter in \(V[G]\), because the factor forcing \(\mathbb P_{[\lambda,\kappa)}\) does not add subsets to \(\lambda\) since \(\lambda\notin \Delta\).

    By Gitik and Kaplan \cite[Prop.\ 2.6 and Prop.\ 2.7]{GITIK_KAPLAN_2023}, since $\lambda$ and $\lambda^+$ remain regular in $V[G]$, it suffices to prove that $j_U[G]$ decides all
    statements of the form $[id]_U\in j_U(\dot{A})$, where $\dot{A}$ is a $\mathbb{P}_\kappa$-name for a subset of $\lambda$. 
    Given such a name $\dot{A}$, we may assume that it is a $\mathbb{P}_\kappa\restriction\lambda$-name. 
    
    Let us work in $V_U[G]$. Let $\dot{A}_0$ be the $\mathbb{P}_{(\lambda,j_U(\lambda))}$-name obtained from $j_U(\dot{A})$. 
    By the Prikry property, there is a condition $q\leq^* 1_{\mathbb{P}_{(\lambda,j_U(\lambda))}}$ such that $q \ ||\  ``[id]_U\in \dot{A}_0"$ and therefore there is $p\in G$ such that
    $$p\Vdash q \ || \ ``[id]_U\in j_U(\dot{A})."$$ In other words, $(p,q) \ || \ ``[id]_U\in j_U(\dot{A})"$. Let us show that there is $p'\in G$ such that $j_U(p')\leq (p,q)$.
    Fix a sequence \(\langle A_\gamma : \gamma \in j_U(\Delta)\cap (\lambda,j_U(\lambda)\rangle\in V_U\) such that for all \(\gamma\in \dom(q)\), 
    $q(\gamma) = (\emptyset,A_\gamma)$. Let $$\alpha\mapsto \l A^{(\alpha)}_\gamma\mid \gamma\in \Delta\cap[\pi(\alpha),\lambda)\r$$ represent $q$ in $j_U$ where $\pi$ is the function representing $\lambda$ in $V_U$. 
    
    Let us proceed with a density argument. Let $p\in\mathbb{P}_\kappa$. We shrink $p$ in the interval $[\gamma_0,\lambda)$, where $\gamma_0=\max(b_p\cap\lambda)+1$. For $\gamma\geq\gamma_0$, $\gamma\in\Delta\cap\lambda$, define $A^0_\gamma=\Delta_{\delta<\gamma}A^{(\delta)}_\gamma$.  Moreover, given that $f$ witnesses that $U$ is $(\Delta\cap\lambda)$-mild, if $f(\sup(\Delta\cap \gamma))\geq\gamma$ we let $A^*_\gamma=A^0_\gamma \setminus \sup(\Delta\cap \gamma)$ and if $f(\sup(\Delta\cap\gamma))<\gamma$ we let $$A^*_\gamma=A^0_\gamma\setminus \max\{ f(\sup(\Delta\cap\gamma)),\sup(\Delta\cap\gamma)\}.$$ We shrink $p$ so that $A^p_\gamma\subseteq A^*_\gamma$ for $\gamma\in \Delta\setminus \gamma_0$, and by genericity we can find such a $p^*\in G$ such that $p^*\leq p$. Then
$j_U(p^*)\restriction\lambda =p^*\leq p$ and $j_U(p^*\restriction [\lambda,\kappa))\leq j_U(p\restriction [\lambda,\kappa))$. For $\gamma\in j_U(\Delta)\cap [\lambda,j_U(\kappa))$, we note that $\lambda\notin j_U(\Delta)$ by our assumption that $\lambda\notin\Delta$ and $\Delta$ is discrete. We need to argue that $A^{j_U(p^*)}_\gamma\subseteq A_\gamma$.  
Denote by $$j_U(\alpha\mapsto \l A^{(\alpha)}_\gamma\mid \gamma\in \Delta\cap[\pi(\alpha),\lambda)\r)(\beta)=\l A_{\beta,\gamma}\mid \gamma\in j_U(\Delta)\cap [j_U(\pi)(\beta),j_U(\lambda))\r.$$
Then by definition, $A_{[id]_U,\gamma}=A_\gamma$ for every $\gamma\in j_U(\Delta)\cap [\lambda,j_U(\lambda))$. If $\gamma=\min(j_U(\Delta)\setminus\lambda)$, then $j_U(f)(\sup(j_U(\Delta)\cap\lambda))=j_U(f)(\lambda)\geq[id]_U$ and we have $$A^{j_U(p^*)}_\gamma\subseteq\Delta_{\beta<\gamma}A_{\beta,\gamma}\setminus j_U(f)(\lambda)+1\subseteq A_{[id]_U,\gamma}=A_\gamma.$$ If $\gamma>\min(j_U(\Delta)\setminus\lambda)$, then $\sup(j_U(\Delta)\cap\gamma)\geq\min(j_U(\Delta)\cap\gamma)\geq [id]_U$, hence we still have $A^{j_U(p^*)}_\gamma\subseteq A_\gamma$. It follows that $j_U(p^*)\leq p^\smallfrown q$ as wanted.
\end{proof}

    The discreteness of the set $\Delta$ is essential here; without it, even $0$-order normal measures might not generate ultrafilters. 
    We give an example of this behavior due to Ben-Neria in Example~\ref{example 1}.

\begin{example}\label{exmaple 2}
    In Section~\ref{subSection: two normals}, we will give an example where $U$ does not generate an ultrafilter in $V[G]$ even though $[id]_U<\min(j_U(\Delta)\setminus \kappa)$. Hence the assumption regarding the existence of the function $f$ cannot be removed in Theorem \ref{Thm: Delta-mild generates an ultraiflter}.
    
    Note that if $[id]_U$ is below the first  $M_U$-Ramsey cardinal (for example) above $\lambda$, then $U$ is $\Delta$-mild. Moreover, by the counterexample of Section~\ref{subSection: two normals}, $[id]_U$ being below the first $M_U$-measurable cardinal does not suffice for $U$ to be $\Delta$-mild, or even for \(U\) to generate an ultrafilter in \(V[G]\).
\end{example}\section{The complete iteration}\label{Section: Complete iteration}
   One aspect of understanding the lifts $\overline W$ of an ultrafilter $W$ to a generic extension \(V[G]\) is to analyze the factor map $k : V_W\to j_{\overline W}^{V[G]}(V)$ defined by $k([f]_W)=[f]_{W^*}^{V[G]}$. In the case of a generic extension by an iteration of Prikry forcings, the factor map can often be described in terms of the complete iteration:
\begin{definition}\label{def: complete iteration}
    Let $\vec{U}=\l U_\delta\mid \delta\in \Delta\r$ be a sequence such that $U_\delta$ is a normal measure on $\delta$ and $\Delta$ is any set of measurables. Let us define, by transfinite recursion, the \textit{complete iteration} associated with the sequence \(\vec U\).
    This will be an iterated ultrapower of \(V\) denoted by $$\l N_\alpha, i^{(\vec{U})}_{\alpha,\beta}\mid \alpha\leq\beta\leq \theta\r$$ 
    starting with \(N_0 = V\).
    Simultaneously, we will define sets of ordinals $\l s^\alpha_\delta\mid \delta \in i_{0,\alpha}(\Delta)\r$ starting with \(s^0_\delta = \emptyset\) for all \(\delta\in \Delta\).
    
    For \(\alpha < \theta\), let $W_\alpha = i_{\alpha}(\vec{U})_{\delta_\alpha}$ where $\delta_\alpha\in i_\alpha(\Delta)$ is the minimal $\delta$ such that the set $s^\alpha_\delta$ is finite. Let $N_{\alpha+1} = (N_\alpha)_{W_\alpha}$, and let $i_{\alpha,\alpha+1}:N_\alpha\rightarrow N_{\alpha+1}$ be the ultrapower embedding associated with \(W_\alpha\). Define $$s^{\alpha+1}_\beta=\begin{cases} \emptyset & \beta\in i_{\alpha+1}(\Delta)\setminus i_{\alpha,\alpha+1}[\Delta]\\
   s^\alpha_{\delta_\alpha}\cup\{\delta_\alpha\} & \beta=i_{\alpha,\alpha+1}(\delta_\alpha)\\
    s^\alpha_\beta & o.w.\end{cases}.$$ At limit steps $\rho$, we take the direct limits of the embeddings and let $s^\rho_\beta=\bigcup\{ s^\alpha_{\beta'}\mid \alpha<\rho, i_{\alpha,\rho}(\beta')=\beta\}$ for every $\beta\in i_{\rho}(\Delta)$. The sequence $\l s^\theta_\alpha\mid \alpha\in i_{0,\theta}(\Delta)\r$ is called the \textit{sequence of sets of indiscernibles} associated with the complete iteration of \(\vec U\).
\end{definition}
We usually omit the superscript $\vec{U}$ from $i^{\vec{U}}$, and we denote $\theta$ by $\infty$ when the length of the iteration does not play a significant role.

For example, if $\Delta$ consists of a single measurable cardinal, then the complete iteration is simply the $\omega^{\text{th}}$ iterated ultrapower of the only normal measure in $\vec{U}$, and the sequence of sets of indiscernibles consists of the corresponding sequence of iteration points. 
\begin{proposition}
    Let $i:V\rightarrow N$ be the complete iteration of $\vec{U}$, where $\dom(\vec{U})=\Delta$ is discrete.  Then:
    \begin{enumerate}
        \item For every $\alpha$, $\delta_{\alpha+1}=i_{\alpha,\alpha+1}(\delta_\alpha)$.
        \item For each $\gamma\in \Delta$, there is a unique $\alpha=\alpha(\gamma)$  such that $\delta_\alpha=\gamma$.
        \item For each $\gamma\in \Delta$, $i_{0,\alpha}(\gamma)=\gamma$, and $i(\gamma)=i_{\alpha,\alpha+\omega}(\gamma)$, where $\alpha=\alpha(\gamma)$.
        
    \end{enumerate}
\end{proposition}
\begin{proof}
For $(1)$, we have that $i_\alpha(\Delta)\cap \delta_\alpha$ is bounded in $\delta_\alpha$, and therefore $$i_{\alpha+1}(\Delta)\cap i_{\alpha,\alpha+1}(\delta_\alpha)=i_\alpha(\Delta)\cap \delta_\alpha.$$
    Hence every $\delta\in i_{\alpha+1}(\Delta)\cap i_{\alpha,\alpha+1}(\delta_{\alpha})$ has been used infinitely many times and $\delta_{\alpha+1}\geq i_{\alpha,\alpha+1}(\delta_\alpha)$. By definition of 
    $\delta_\alpha$, equality must hold.
    
   For $(2)$ and $(3)$, fix any $\gamma\in\Delta$, then $\gamma$ is measurable, it is routine to show  inductively on $\beta$, if $\delta_\beta<\gamma$, then $i_{0,\beta}(\gamma)=\gamma$. Since the sequence $\delta_\alpha$ is strictly increasing, and since $\Delta$ is discrete, after less than $\gamma$-many steps of the iteration we reach a stage $\alpha$, such that $\gamma=\delta_\alpha$. This means that the measure $i_{0,\alpha}(\vec{U})_\gamma$ was not used before stage $\alpha$, and every previous measure was used infinitely many times already. By $(1)$, $\delta_{\alpha+n}=i_{\alpha,\alpha+n}(\gamma)$, for every $n<\omega$. At stage $\alpha+\omega$ we get:
    $$i_{\alpha,\alpha+\omega}(i_{0,\alpha}(\Delta))=[i_{0,\alpha}(\Delta)\cap\gamma]\cup i_{\alpha,\alpha+\omega}(i_{0,\alpha}(\Delta)\setminus \gamma)$$
    So the minimal element which can potentially be applied now is $i_{\alpha,\alpha+\omega}(\gamma)$, however, we have already used this measurable $\omega$-many times, so by the definition of the complete iteration we must go to the next element which makes the critical point of $i_{\alpha+\omega,\theta}$ greater than $i_{\alpha,\alpha+\omega}(\gamma)$. We conclude that $i_{\alpha+\omega,\theta}(i_{\alpha,\alpha+\omega}(\gamma))=i_{\alpha,\alpha+\omega}(\gamma)$ and therefore $i(\gamma)=i_{\alpha,\alpha+\omega}(\gamma)$ as wanted. Also we see that for every $\alpha$, $$i(\Delta)=i_{0,\alpha}(\Delta)\cap \delta_\alpha\cup\{i_{\alpha,\alpha+\omega}(\delta_\alpha)\}\cup i(\Delta)\setminus i_{\alpha,\alpha+\omega}(\delta_\alpha)$$
    
%    Now suppose toward contradiction that $i''(\Delta)\subsetneq i(\Delta)$ and let $\lambda\in i(\Delta)$ be the minimal not in $i''(\Delta)$. Since $\Delta$ is discrete, $\lambda'=\sup(i(\Delta)\cap\lambda)<\lambda$. Let $\delta'=\sup i^{-1}[\lambda']$, then by minimality of $\lambda$, $i''\Delta\cap \delta'=i(\Delta)\cap \lambda'$. Let $\delta_0=\min(\Delta\setminus \delta')\geq\delta'$, then $\$
\end{proof}
\begin{definition}\label{Definition: Generic}
    Given a sequence of sets of ordinals $\vec{s}=\l s_\alpha\mid \alpha\in \Delta\r $, each of order type $\omega$, we associate a filter $G_{\vec{s}}$ on $\mathbb{P}_\kappa$ consisting of all conditions $p$ such that for each $\alpha$, $t^p_\alpha$ is an initial segment of $s_\alpha$ and $s_\alpha\setminus t^p_\alpha\subseteq A^p_\alpha$.
\end{definition}
Suppose \(i : M \to N\) is elementary, $\mathbb{P}\in M$ is a poset, and 
\(G\subseteq i(\mathbb P)\) is \(N\)-generic. Then 
\[\hull^{N[G]}(i[M]\cup \{G\}) = \{i(\tau)_G : \tau\in M^{\mathbb P}\}.\]
Then \(\hull^N(i[M]\cup \{G\})\) is the smallest elementary substructure of \(N[G]\) containing \(i[M]\cup \{G\}\). That it is an elementary substructure can be proved by verifying the Tarski-Vaught criterion using the Fullness Lemma in forcing; see \cite[Corollary 25]{hamkins2012wellfoundedbooleanultrapowerslarge} for a different approach using Boolean ultrapowers. 
\begin{lemma}\label{Lemma: Generic}
Let $\vec{U}$ be a sequence of normal measures on a discrete set $\Delta$ with $\sup(\Delta)=\kappa$. Let $i=i_{0,\theta}:V\rightarrow N$ be the complete iteration of $V$ by $\vec{U}$, and $\vec{s}=\l s^\theta_\alpha\mid \alpha \in i(\Delta)\r$ be the sequence of sets of indiscernibles, and let \(G = G_{\vec s}\). Then the following hold:
\begin{enumerate}
    \item $G$ is $N$-generic for $i(\mathbb{P}_\kappa)$. 
    \item If $p$ is a pure condition then $i(p)\in G$.
    \item \(N[G] = \hull^{N[G]}(i[V]\cup \{G\})\).
    \item $N[G]$ is closed under \(\delta\)-sequences where \(\delta = \min(\Delta)\).
\end{enumerate}
\end{lemma}
\begin{proof}
    The genericity of $G$ is due to Fuchs \cite{Fuchs2005-FUCACO}.
    For the second item, if $p$ is pure, then by induction we can prove that for every $\xi\leq\theta$, $s^\xi_\alpha\subseteq A^{i_{0,\alpha}(p)}_\alpha$. This implies in particular that $s^\theta_\alpha\subseteq A^{i(p)}_\alpha$, as wanted.  Suppose this was true for $\xi$, since $i_{\xi,\xi+1}$ is the ultrapower by the normal $W_\xi$ on $\delta_\xi$, and $A^{i_{0,\xi}(p)}\in W_\xi$, we conclude that $\delta_\xi\in A^{i_{\xi,\xi+1}(p)}_{\delta_\xi}$. Since below $\delta_\xi$, things do not change, we get that $s^{\xi+1}_{i_{\xi,\xi+1}(\delta_\xi)}\subseteq A^{i_{\xi,\xi+1}(p)}_{\delta_\xi}$. Note that above $\xi$, all the sequences are empty at this stage of the iteration. 

    For the third item, let $H = \hull^{N[G]}(i[V]\cup \{G\})$.
    To establish that $H = N[G]$, we will show that \(N\subseteq H\). Then \(H = N[G]\) since $H\prec N[G]$ and \(N\cup \{G\}\in H\).
    
    Let \(S = \bigcup\vec{s}\), the set of critical points of the complete iteration. Then \(S\) is itself discrete and the set of accumulation points \(\text{acc}(S)\) equals \(\text{acc}(i(\Delta))\). 
    Since every element of $N$ is of the form $i(f)(\kappa_1,...,\kappa_m)$ for some $f:[\kappa]^m\rightarrow V$, where $\kappa_1,...,\kappa_m\in S$, it suffices to prove that $S\subseteq H$. In turn, it suffices to prove that $i(\Delta)\subseteq H$, since any element in $S$ is a Prikry point associated with one of the elements of $i(\Delta)$ and therefore is definable from $G$ and $i(\Delta)$ in \(N[G]\).
    
    Suppose towards contradiction that $\zeta$ is the minimal element of $i(\Delta)$ not in $H$. Since $i(\Delta)$ is discrete, $\gamma=\sup(i(\Delta)\cap\zeta)<\zeta$ . Note that \(\gamma\notin H\), since otherwise easily \(\zeta\in H\).
    It follows that \(\gamma\in \text{acc}(i(\Delta))\) since otherwise \(\gamma\in i(\Delta)\), and hence \(\gamma\in H\) by the minimality of \(\zeta\). 
    Note that \(\gamma\subseteq H\). 
    Note that $\gamma\notin S$ since \(S\) is discrete and $ \gamma\in\text{acc}(i(\Delta))\subseteq \text{acc}(S)$.
    Therefore $\gamma=i(f)(\kappa_1,...,\kappa_m)$ for some $\kappa_1,...,\kappa_m<\gamma$ and so $\gamma\in H$, which is a contradiction.
    
    The final item follows from the previous one using the standard proof that the ultrapower of \(V\) by a \(\delta\)-complete ultrafilter is closed under \(\delta\)-sequences.
    (In fact, one can view \(N[G]\) as a Boolean ultrapower of \(V\) and apply \cite[Theorem 28]{hamkins2012wellfoundedbooleanultrapowerslarge}.)
    Fix a \(\delta\)-sequence of elements \(\langle x_\alpha\mid \alpha < \delta\rangle\subseteq N[G]\). 
    By the previous item, 
 each $x_\alpha$ is the interpretation of some $i(\tau_\alpha)$ via $G$. Note that $$\l i(\tau_\alpha)\mid \alpha<\delta\r=i(\l \tau_\alpha\mid \alpha<\delta\r)\restriction\delta\in N.$$ 
 Therefore $\l x_\alpha\mid \alpha<\delta\r\in N[G]$.
\end{proof}
The following proposition shows that any ultrafilter used in the complete iteration gives rise to an elementary embedding of the target model of the complete iteration.
\begin{proposition}\label{Proposition: indiscernible after ultrapower}
    If \(i=i_{0,\theta} : V\to N\) is the complete iteration \(V\) by \(\vec U\), \(\vec s\) is the associated sequence of sets of indiscernibles, and \(U\) is an ultrafilter on \(\delta\) applied at some stage of the complete iteration, then \(j_U^{N[G_{\vec s}]}\circ i = i\)
    and \(j_U^{N[G_{\vec s}]}(\vec s\hspace{.05cm}) = \vec{s}\setminus \{\delta\}\).
    \end{proposition}
    \begin{proof}
        We first consider the special case that \(\delta\) is the least element of the domain of \(\vec{U}\).
        Note that \(j_U^{N[G_{\vec s}]} = j_U^V\restriction N[G_{\vec s}]\) because \(N[G_{\vec s}]\) is closed under \(\delta\)-sequences. Therefore \[j_U^{N[G_{\vec s}]} \circ i = j_U\circ i = j_U(i)\circ j_U = i^U\circ j_U\] 
        where \(i^U : V_U\to N\) denotes the complete iteration of \(V_U\) by \(j_U(\vec U)\). By the definition of the complete iteration, \(i^U\circ j_U = i\),
        which proves that \(j_U^{N[G_{\vec s}]}\circ i = i\).
        Similarly, \(j_U^{N[G_{\vec s}]}(\vec s)\) is the sequence of sets of indiscernibles associated with the complete iteration of \(V_U\) by \(j_U(\vec U)\), which is \(\vec{s}\setminus \{\delta\}\).

        To prove the claim in general, suppose \(U\) is the ultrafilter used at stage \(\alpha\) of the complete iteration. Then \(i = i_{\alpha\infty}\circ i_{0\alpha}\) where \(i_{0\alpha} : V\to N_\alpha\) is an initial segment of the complete iteration \(V\) by \(\vec U\) and \(i_{\alpha\infty} : N_\alpha\to N\) is the complete iteration of \(N_\alpha\) by \(\vec W = i_{0\alpha}(\vec U)\restriction [\delta,\infty)\). Let \(\vec s^{\hspace{.07cm}0\alpha}\) denote the sequence of sets of indiscernibles associated with \(i_{0\alpha}\) and let \(\vec s^{\hspace{.07cm}\alpha\infty}\) denote the sequence of sets of indiscernibles associated with \(i_{\alpha\infty}\). Then \(G_{\vec s^{\hspace{.04cm}0\alpha}}\times G_{\vec s^{\hspace{.04cm}\alpha\infty}} = G_{\vec s}\).
        
        The special case of the claim proved above, applied in \(N_\alpha\) to \(U\) and \(\vec W\), implies that
        \(j_U^{N[G_{\vec s^{\hspace{.04cm}\alpha\infty}}]}\circ i_{\alpha\infty} = i_{\alpha\infty}\) and
        \(j_U^{N[G_{\vec s^{\hspace{.04cm}\alpha\infty}}]}(\vec s^{\hspace{.07cm}\alpha\infty}) = \vec s^{\hspace{.07cm}\alpha\infty}\setminus \{\delta\}\).
        This yields the full claim almost immediately, except that we have to verify that 
        \(j_U^{N[G_{\vec s^{\hspace{.04cm}\alpha\infty}}]}\)
        is the restriction of  \(j_U^{N[G_{\vec s}]}\) to \(N[G_{\vec s^{\hspace{.04cm}\alpha\infty}}]\).
        But
        \(j_U^{N[G_{\vec s}]} = j_U^{N[G_{\vec s^{\hspace{.04cm}0\alpha}}\times G_{\vec s^{\hspace{.04cm}\alpha\infty}}]}\), which extends \(j_U^{N[G_{\vec s^{\hspace{.04cm}\alpha\infty}}]}\)
        by the proof of the classical L\'evy-Solovay theorem  \cite{Levy1967MeasurableCA}. 
    \end{proof}

    \subsection{The canonical extension of an ultrafilter}\label{Section: Canonical extension}
Suppose that in \(V\), \(W\) is a \(\kappa\)-complete ultrafilter on a set \(X\). Using a construction due to Mitchell \cite{MithcellSkies}, we will define
a \(\kappa\)-complete $V[G]$-ultrafilter \(W^*\in V[G]\)
extending \(W\). (The extension \(W^*\) is essentially due to Magidor; it is the construction of \(W^*\) using complete iterations that is due to Mitchell.)

Let \(j : V\to M\) be the ultrapower embedding associated with \(W\). Let \(\mathbb Q = \mathbb P_{>\kappa}\),
so that \(j(\mathbb P_\kappa)\) can be naturally identified with \(\mathbb P_\kappa\times \mathbb Q\). %Do we discuss the complete iteration as a product?
To lift \(j\) to \(V[G]\), one might try to produce an \(M\)-generic \(H\) on \(\mathbb Q\) such that \(j[G]\subseteq G\times H\).
If \(\Delta\) is unbounded in \(\kappa\), however, no such \(M\)-generic \(H\) can exist in \(V[G]\): the ordinal \(\delta= \min(j(\Delta)\cap (\kappa,j(\kappa)))\) has uncountable cofinality in \(V[G]\), so in $V[G]$, there is no cofinal \(\omega\)-sequence in \(\delta\).

Instead, we will produce an inner model $N$ and an elementary embedding \(i: M\to N\) such that there is an \(N\)-generic filter \(H\) on \(i(\mathbb Q)\) with \(i\circ j[G]\subseteq G\times H\). Then by the Silver lifting criterion, $i\circ j$ lifts to an elementary embedding from $V[G]$ to $N[G\times H]$.

Let \(i : M \to N\) be the complete iteration of
\(M\) via the sequence \(j(\vec U)\restriction (\kappa,j(\kappa))\).  Let \(\vec s\) be the sequence
of sets of indiscernibles associated with the complete iteration,
and let \(G_{\vec s}\) be the associated \(N\)-generic filter on \(i(\mathbb Q)\) (Definition \ref{Definition: Generic}). %Establish notation for the factor forcing P_{\alpha,\beta}

Note that \(G\) is an \(N[G_{\vec{s}}]\)-generic filter on \(\mathbb P_\kappa\), and so 
by the Product Lemma, \(G\times G_{\vec{s}}\) is an \(N\)-generic filter on \(\mathbb P_\kappa\times \mathbb Q\).
We must verify that \(i\circ j[G]\subseteq G\times G_{\vec{s}}\). 
For each \(p\in \mathbb P_\kappa\),
\(j(p)\) has the form \((p,q)\)
where \(q\in \mathbb Q\) is a \textit{pure condition}, or in other words, a direct extension of
\(1_\mathbb Q\), and therefore \(i\circ j(p) = (p,i(q))\).
For every pure condition \(q\in \mathbb Q\), \(i(q)\) belongs to \(H\),
by Lemma  \ref{Lemma: Generic}. Hence $i\circ j(p)=(p,i(q))\in G\times G_{\vec{s}}$.

By the Silver lifting criterion, let  \(j^* : V[G]\to N[G\times G_{\vec{s}}]\) be the unique lift of $i\circ j$ such that \(j^*(G) = G\times G_{\vec{s}}\). 
\begin{definition}\label{Definition: Canonical extension}
    The \textit{canonical extension} of \(W\) to \(V[G]\) is the \(V[G]\)-ultrafilter \(W^*\) on \(X\) derived from \(j^*\) using \(i([\text{id}]_W)\).
\end{definition}
Note that \(W^*\) is \(\kappa\)-complete in \(V[G]\) since it is derived from \(j^*\), which has critical point at least \(\kappa\). 

The next proposition shows   \(N[G\times G_{\vec{s}}] = V[G]_{W^*}\) and \(j^* = j_{W^*}\).
\begin{proposition}\label{Prop: embedding of W*}
    Suppose  \(W\) is a \(\kappa\)-complete ultrafilter and \(W^*\) is the canonical extension of \(W\) to \(V[G]\). Then \(j_{W^*}\restriction V = i\circ j_W\) where
    \(i : V_W\to N\) is the complete iteration of \(V_W\) associated with \(j_W(\vec U)\restriction (\kappa,j_W(\kappa))\). Moreover, $j_{W^*}(G)=G\times G_{\vec{s}}$ where $\vec{s}$ is the sequence of sets of indiscernibles associated with the complete iteration. 
\end{proposition}
\begin{proof}
    The proposition follows from Corollary \ref{cor: hulls and ultrapowers} once we show that $$N[G\times G_{\vec{s}}]=\hull^{N[G\times G_{\vec{s}}]}(i\circ j_W[V]\cup \{G, G_{\vec{s}},i([id]_W)\}).$$
    This in fact implies directly that $j_{W^*}=j^*$. To see the above, applying the third item of Lemma \ref{Lemma: Generic} inside of $M$, we  see that 
     $$N[G_{\vec{s}}]=\hull^{N[ G_{\vec{s}}]}(i\circ j_W[V]\cup \{G_{\vec{s}},i([id]_W)\}).$$
     Then we finish by noting that being a forcing extension of $N[G_{\vec{s}}]$, $N[G\times G_{\vec{s}}]$ is the hull (in itself) of $N[G_{\vec{s}}]\cup\{G\}$.
\end{proof}

In the case that \(W\) is $\Delta$-mild, Theorem \ref{Thm: Delta-mild generates an ultraiflter} shows that \(W^*\) is the unique extension of \(W\) in \(V[G]\) and therefore Proposition \ref{Prop: embedding of W*} can be used to analyze the ultrapower of the ultrafilter generated by $W$ in the generic extension. This is a slight generalization of Kaplan's theorem stated in the introduction. Although this is just the combination of Theorem \ref{Thm: Delta-mild generates an ultraiflter} and Proposition \ref{Prop: embedding of W*}, we record it for future use.
\begin{corollary}\label{Cor: Kaplan's Theorem}
Let $W$ be a $\Delta$-mild ultrafilter in $V$ and let $G\subseteq \mathbb{P}_\kappa$ be $V$-generic. In $V[G]$, $W$ generates its own canonical extension $W^*$ and $j_{W^*}\restriction V=i\circ j_W$, where
    \(i : V_W\to N\) is the complete iteration of \(M\) associated with \(j(\vec U)\restriction (\kappa,j(\kappa))\). Moreover, $j_{W^*}(G)=G\times G_{\vec{s}}$ where $\vec{s}$ is the sequence of sets of indiscernibles associated with the complete iteration.
    
\end{corollary}
We now give Ben-Neria's counterexample to normal measures generating ultrafilters after (non-discrete) Magidor iterations.
\begin{example}\label{example 1}
    Suppose \(\kappa\) is a measurable cardinal of Mitchell order \(2\), \(\Delta\) is the set of all measurable cardinals less than \(\kappa\), and \(G\subseteq \mathbb P_\kappa\) is \(V\)-generic for the Magidor iteration.
    
    Take any normal measure $U_1$ over $\kappa$ of Mitchell order 1 in the ground model. 
    Ben-Neria \cite[Notation 2.3]{Ben-Neria2014-BENFMI} defines an ultrafilter \(U_1^\times\) as
    the projection of \(U_1^*\) by the function \(d : \Delta\to \kappa\) mapping each measurable cardinal in \(\Delta\) 
    to the least element of its corresponding Pr\'ikry sequence.
    
    By Ben-Neria \cite[Proposition 2.1]{Ben-Neria2014-BENFMI}, $U_1^\times$ is a normal measure in $V[G]$. Let $U_0=U_1^\times\cap V$, so that $U_0$ is a normal ultrafilter in $V$ by \cite{GITIK_KAPLAN_2023}. Also $U_0$ must be of order $0$ (and in particular \(U_0 \neq U_1\)): otherwise, \(d\) would be a finite-to-one regressive function defined \(U_1^\times\)-large set. 
    
    Consider the ultrafilter $U_0^*$. Then $U_0^*$ is a normal ultrafilter \cite[2.1]{Ben-Neria2014-BENFMI} in $V[G]$ and we claim that $U_0^*\neq U_1^\times$, which yields two distinct extensions of $U_0$ to normal ultrafilters in $V[G]$. Indeed, since $U_1$ is of order $1$, the set $\Delta$ of measurables below $\kappa$ is in $U_1$, and therefore, by definition $d[\Delta]\in d_*(U_1)=U_1^\times$. On the other hand, as $U_0$ is of order $0$, $d[\Delta]\notin U_0^*$:
    otherwise, \(\kappa\) must appear in the Pr\'ikry sequence in \(V[G]_{U_0^*}\) associated to some element of \(j_{U_0^*}(\Delta)\), but by Proposition \ref{Prop: embedding of W*}, the elements of these Pr\'ikry sequences are critical points in the complete iteration of \(V_{U_0}\); \(\kappa\) cannot be such a critical point since \(\kappa\) is not measurable in \(V_{U_0}\).
    Hence $U_1^\times\neq U_0^*$.
\end{example}

\subsection{Mitchell's lemma} Next, we prove a variation of Mitchell's lemma from \cite{MithcellSkies}, which gives an analysis of ultrapowers of canonical lifts within complete iterates of \(V\) that is a key part of the proof of our classification results in Section \ref{subsection: classification}. 
The idea is to take a sufficiently complete ultrafilter \(W\) and a complete iteration \(i : V \to N\) and analyze the ultrapower embedding associated to the canonical extension of \(i(W)\) to \(N[G_{\vec s}]\), where \(\vec s\) is the sequence of sets of indiscernibles. It turns out that this embedding is simply the restriction to \(N[G_{\vec s}]\) of the ultrapower embedding \(j_W :  V\to V_W\).

Let $\vec{U}$ be a sequence of normal measures on a discrete set $\Delta\subseteq \kappa$. Let $i=i_{0,\infty}:V\rightarrow N$ be the complete iteration of $V$ by $\vec{U}$, and $\vec{s}=\l s^\infty_\alpha\mid \alpha \in i(\Delta)\r$ be the associated sequence of sets of indiscernibles, and let \(G = G_{\vec s}\). 

     \begin{lemma}[Mitchell]\label{Lemma: Mitchell}
        Suppose that $W$ is a \(\kappa\)-complete ultrafilter. Then \[j_{i(W)^*}^{N[G_{\vec s}]} = j_{W}\restriction N[G_{\vec{s}}].\] Moreover, the canonical factor map $k:N_{i(W)}\to j_{i(W)^*}(N)$ is given by the complete iteration of $N_{i(W)}$ by $i(j_W(\vec{U}))\restriction (\kappa, j_W(\kappa))$.  
         \end{lemma}% https://q.uiver.app/#q=WzAsNyxbMCwyLCJOIl0sWzAsMywiViJdLFsyLDIsIk5fe2koVyl9Il0sWzIsMCwiTltHX3tcXHZlY3tzfX1dX3tpKFcpXip9Il0sWzIsMSwial97V30oTikiXSxbMiwzLCJWX3tXfSJdLFswLDEsIk5bR197XFx2ZWN7c319XSJdLFswLDIsImpfe2koVyl9Il0sWzIsNCwiaV57al9XKFxcdmVje1V9KX1fe1xca2FwcGEsal9XKFxca2FwcGEpfSIsMl0sWzAsNCwial9XIl0sWzQsMywiIiwwLHsic3R5bGUiOnsiYm9keSI6eyJuYW1lIjoiZGFzaGVkIn19fV0sWzUsMiwiaSIsMl0sWzEsNSwial97V30iXSxbMSwwLCJpPWlee1xcdmVje1V9fSJdLFswLDYsIiIsMCx7InN0eWxlIjp7ImJvZHkiOnsibmFtZSI6ImRhc2hlZCJ9fX1dLFs2LDMsImpfe2koVyleKn0iXSxbNSw0LCJpXntqX1coXFx2ZWN7VX0pfSIsMix7ImN1cnZlIjo1fV1d
\begin{figure}
    \centering
\[\begin{tikzcd}
	&& {N[G_{\vec{s}}]_{i(W)^*}} \\
	{N[G_{\vec{s}}]} && {j_{W}(N)} \\
	N && {N_{i(W)}} \\
	V && {V_{W}}
	\arrow["{j_{i(W)^*}}", from=2-1, to=1-3]
	\arrow[phantom, sloped,"\subseteq", from=2-3, to=1-3]
	\arrow[phantom, sloped,"\subseteq", from=3-1, to=2-1]
	\arrow["{j_W}", from=3-1, to=2-3]
	\arrow["{j_{i(W)}}", from=3-1, to=3-3]
	\arrow["{i^{j_W(\vec{U})}_{\kappa,\infty}}"', from=3-3, to=2-3]
	\arrow["{i=i^{\vec{U}}}", from=4-1, to=3-1]
	\arrow["{j_{W}}", from=4-1, to=4-3]
	\arrow["{i^{j_W(\vec{U})}}"', curve={height=70pt}, from=4-3, to=2-3]
	\arrow["i"', from=4-3, to=3-3]
\end{tikzcd}\]
    \caption{Mitchell's diagram}
    \label{fig:Mitchell Lemma}
\end{figure}
\begin{proof}
    By replacing \(\kappa\) with the completeness of \(W\), we may assume that \(\kappa\) is a measurable cardinal, and in particular that \(\kappa\) is strongly inaccessible.

    We first note the bottom rectangle in
    Figure \ref{fig:Mitchell Lemma}
    commutes; this is just by the
    usual formula 
    \(j_{i(W)}\circ i = i \circ j_W\).

    Second, we note that for the same reason the quadrangle with vertices $V$, $N$, $j_W(N)$, and $V_W$ also commutes, using here that $j_W(\vec{U})\restriction\kappa=\vec{U}$ and therefore \[j_W(i) = i^{j_W(\vec U)} = i_{\kappa,\infty}^{j_W(\vec U)}\circ i\]

    For ease of notation, for any \(j : V\to M\) obtained by composing arrows in the diagram, let \(\vec{U}^M = j(\vec U)\). (The fact that this is well-defined follows from the commutativity we have established.)

    The commutativity established in the previous paragraphs allows us to show that
    \[j_W\restriction N = j_{i(W)^*}^{N[G_{\vec {s}}]}\restriction N\]
    To prove this, we first observe
    \(i^{\vec{U}^{N_{i(W)}}\restriction (\kappa,\infty)} =  i_{\kappa,\infty}^{\vec U^{V_W}},\)
    by the definition of the complete iteration: since \(i_{0\kappa}^{\vec U^{V_W}} = i\restriction V_W\), which embeds \(V_W\) into \(N_{i(W)}\), \(i_{\kappa,\infty}^{\vec U^{V_W}}\) is the complete iteration of \(\vec U^{N_{i(W)}}\) restricted to measures above \(\kappa\). 
    
    Then by Proposition \ref{Prop: embedding of W*}, 
    \[j_{i(W)^*}^{N[G_{\vec {s}}]}\restriction N = i^{\vec U^{N_{i(W)}}\restriction (\kappa,\infty)} \circ j_{i(W)} = i_{\kappa,\infty}^{\vec U^{V_W}}\circ j_{i(W)}\]
    By the commutativity facts established in the first two paragraphs, \(j_W\restriction N\) and
    \(i_{\kappa,\infty}^{\vec U^{V_W}}\circ j_{i(W)}^N\) agree on \(i[V]\). Moreover
    \(j_W\restriction N\) and \(i_{\kappa,\infty}^{\vec U^{V_W}}\circ j_{i(W)}^N\) agree on \(\kappa\), both being the identity there. 
    Since \(N = \hull^N(i[V]\cup \kappa)\),
    it follows that \(j_W\restriction N = j_{i(W)^*}^{N[G_{\vec {s}}]}\restriction N\) as claimed.

    Next, to show that \(j_{i(W)^*}^{N[G_{\vec {s}}]} = j_W\restriction N[G_{\vec s}]\),
    it now suffices to prove that these two embeddings agree on \(G_{\vec s}\). 
    By elementarity, \(j_W(G_{\vec s})\) is  equal to the generic associated to the sequence \(j_W(\vec s)\) of sets of indiscernibles coming
    from the complete iteration of \(V_W\) by \(j_W(\vec U)\). But \(j_{i(W)^*}(G_{\vec s})\) is equal to the same thing, by the ``moreover" clause of Proposition \ref{Prop: embedding of W*} and calculations similar to the previous paragraph.
\end{proof}
\begin{corollary}
    Suppose $\lambda\notin \Delta$ is a measurable cardinal and $W$ is a \(\lambda\)-complete ultrafilter on $\lambda$.  Then  $j_{i(W)^*}^{N[G_{\vec s}]}=j_W\restriction N[G_{\vec{s}}]$. 
\end{corollary}
\begin{proof}
    By the previous lemma, it suffices to consider the case that \(\lambda < \sup \Delta\). Also, the previous lemma implies that \(j_{\overline{i}(W)^*}^{\overline{N}[G_{\vec t}]} = j_W\restriction \overline{N}[G_{\vec t}]\) where \(\overline i : V\to \overline N\) is the complete iteration of \(\vec U\) restricted to measures below \(\lambda\) and \(\vec t\) is the associated sequence of sets of indiscernibles. 

    Note however that \(\overline i( W) = i(W)\), and by Lemma \ref{Lemma: Generic} applied in \(\overline N[G_{\vec t}]\), \(N[G_{\vec s}]\) is an inner model of \(\overline N[G_{\vec t}]\) that is closed under \(\lambda\)-sequences in \(\overline N[G_{\vec t}]\). It follows that 
    \[j_{i(W)^*}^{N[G_{\vec s}]} = j_{\overline i(W)^*}^{\overline N[G_{\vec t}]}\restriction N[G_{\vec s}] = j_W\restriction N[G_{\vec s}],\] which proves the corollary. 
\end{proof}

%In this paper, we would like to characterize $\sigma$-complete ultrafilters in a generic extension by $\mathbb{P}_\kappa$. 

\subsection{Normal measures and the complete iteration}
In our classification of lifts of ultrafilters under the discrete Magidor product we will have to classify the possible extensions of the point-wise images of a normal ultrafilter under the complete iteration. The next lemmas will be used for that purpose. The tail filter on an ultrafilter $\rho$ is the filter generated by the sets $\{(\alpha,\rho)\mid \alpha<\rho\}$. 
\begin{lemma}\label{lemma: normal generation}
    If \(i : M\to N\) is an elementary embedding, \(U\in M\) is a normal ultrafilter on \(\delta\),
    and \(N = \hull^N(i[M]\cup i(\delta))\),
    then \(F \cup i[U]\) generates \(i(U)\) where \(F\) denotes the tail filter on \(i(\delta)\).
\end{lemma}
\begin{proof}
    Let $X\in i(U)$. Since $N=\hull^N(i[M]\cup i(\delta))$, there is $f\in M$ and $\eta<i(\delta)$ such that $X=i(f)(\eta)$. Changing $f$ if needed, %f'(\alpha)=\delta if f(\alpha)\notin U and f'(\alpha)=f(\alpha) otherwise
     we may assume that $f:\delta\rightarrow U$. Let $A^*=\Delta_{\alpha<\delta}f(\alpha)$, then by the normality assumption $A^*\in U$. It follows that $i(A^*)\setminus \eta+1\subseteq i(f)(\eta)=X$, and therefore $X$ is in the filter generated by $F\cup i[U]$.  
\end{proof}

\begin{lemma}\label{lemma: hull analysis}
    Suppose \(i : M\to N\) is the complete iteration of \(\vec U\), \(\delta = \min(\dom(\vec U))\),
    and \(\eta\leq i(\delta)\). Let
    \(X = \hull^N(i[M]\cup \eta)\), let \(k : \overline N\to N\) be the inverse of the transitive collapse of $X$, and let \(\overline i = k^{-1}\circ i\). Let \(n \leq \omega\) be least such that \(\eta \leq i_{0n}(\delta)\). Then \(\overline{i} = i_{0n}\) and \(k = i_{n\infty}\).
\end{lemma}
\begin{proof}
    Let $M_n$ be the (transitive) $n^{\text{th}}$-iterate of $M$ in $i$, namely $i_{n,\infty}:M_n\rightarrow N$, and $i_{0,n}:M\rightarrow M_n$. We claim that $X=i_{n,\infty}[M_n]$, from which it follows via uniqueness that $i_{n,\infty}=k$ and therefore $\overline{i}=k^{-1}\circ i=i_{0,n}$. Indeed, any $x\in X$ has the form $i(f)(\xi)$ for some $\xi<\eta\leq i_{0,n}(\delta)$. Each such ordinal can be represented using $\delta,i_{0,1}(\delta),...,i_{0,n-1}(\delta)$ and therefore we may assume that $x=i(f)(\delta,i_{0,1}(\delta),...,i_{0,n-1}(\delta))$ since the critical point of $i_{n,\infty}$ is $i_{0,n}(\delta)$, we have that $$x=i_{n,\infty}(i_{0,n}(f)(\delta,i_{0,1}(\delta),...,i_{0,n-1}(\delta)))\in \rng(i_{n,\infty})$$
    The other inclusion is similar.
\end{proof}

\begin{lemma}\label{lemma: skies}
    Suppose \(i : M\to N\) is an elementary embedding \(\delta\) is a regular cardinal, and \(\delta \leq \eta < i(\delta)\) is such that \(\eta\in i(C)\) for every closed unbounded set \(C\subseteq \delta\) in $M$. Let
    \(X = \hull^N(i[M]\cup \eta)\), let \(k : \overline N\to N\) be the inverse of the transitive collapse, and let \(\overline i = k^{-1}\circ i\). Then \(\overline i(\delta) = \eta\). 
\end{lemma}
\begin{proof}
    Since $k^{-1}$ is just the transitive collapse, it suffices to prove that there are no ordinals $\alpha\in X$ between $\eta$ and $i(\delta)$. Let $\alpha\in X$ be below $i(\delta)$. Then by $\alpha=i(f)(\xi)$ for some $\xi<\eta$. We may assume that $f:\delta\rightarrow\delta$. Let $C_f\subseteq \delta$ be the club of closure points of $f$. Then the assumption of the Lemma, $\eta\in i(C_f)$, namely, $\eta$ is a closure point of $i(f)$. Since $\xi<\eta$, $\alpha=i(f)(\xi)<\eta$. Hence no ordinal in $X$ is between $\eta$ and $i(\delta)$.
\end{proof}
A filter on $\rho$ extending the tail filter on $\rho$ is called \textit{fine}.
\begin{lemma}\label{lemma: tilde U}
    Suppose \(M\) is a transitive model of set theory, \(i : M \to N\) is the complete iteration of \(\vec U\) in \(M\), \(U\) is a normal $M$-ultrafilter on $\delta$, where $\delta$ is the minimal element in $\dom(\vec{U})$, and \(\tilde{U}\) is a fine \(N\)-ultrafilter on an ordinal \(\eta\leq i(\delta)\) extending \(\{i(A)\cap \eta : A\in U\}\). 
    Then one of the following holds:
    \begin{itemize}
        \item \(\{i_{0n}(\delta)\}\in \tilde{U}\) for some \(n < \omega\),
        \item \(\tilde{U} = i_{0n}(U)\) for some \(n \leq \omega\).
    \end{itemize}
    
\end{lemma}
    \begin{proof}
        If \(\eta = \nu + 1\) is a successor ordinal, then since \(\tilde{U}\) is fine, \(\{\nu\}\in \tilde{U}\). In that case, for every $A\in U$, $\nu\in i(A)$. Note that since the critical point of $i_{\omega,\infty}$ is greater than $i_{0,\omega}(\delta)$, $i_{0,\omega}(A)=i(A)$. Since $i_{0,\omega}$ is the $\omega^{\text{th}}$ iterate of the normal measure $U_\delta$, the only seeds for a normal ultrafilter in $i_{0,\omega}$ are the $i_{0,n}(\delta)$'s\footnote{If $W$ is a normal ultrafilter derived from $i_{0,\omega}$ using $\alpha$, then $W$ is derived from $i_{0,n}$ and $\alpha'$ for some $n<\omega$ and some $\alpha'$ such that $i_{n+1,\omega}(\alpha')=\alpha$. Therefore $W\leq_{RK} U_\delta^n$. But then $W\equiv_{RK} U^m$ for some $m\leq n$ (see for example \cite[Lemma 2.4]{TomMoti}). Since $W$ is normal, $m=1$ and $W=U$, and again by normality, the only RK-projections of $U^n$ onto $U$ are given by the coordinate projections. Equivalently, the only seeds $\alpha$ for $U$ from $j_{U^n}$ are $\{i_{0,m}(\delta)\mid m<n\}$.}, hence \(\nu = i_{0n}(\delta)\) for some \(n < \omega\). 
        
        Now assume that \(\eta\) is a limit ordinal.
        Since \(\tilde{U}\) is fine, 
        every set in \(\tilde{U}\) is unbounded in \(\eta\). 
        It follows that for every closed unbounded set \(C\subseteq \delta\) in \(M\), \(i(C)\cap \eta\) is unbounded in \(\eta\), and therefore if \(\eta < i(\delta)\), then \(\eta \in i(C)\). 
        Let \(X = \text{Hull}^N(i[M]\cup \eta)\),
        let \(k : \overline N\to N\) be the inverse of the transitive collapse of \(X\), and let
        \(\overline i = k^{-1}\circ i\).
                
        Note that \(\overline i : M\to \overline N\), \(\overline N = \text{Hull}^{\overline N}(\overline i[M]\cup \eta)\), and \(\overline i(\delta) = \eta\) by Lemma \ref{lemma: skies}\footnote{Although in Lemma \ref{lemma: skies} we are assuming that $\eta<i(\delta)$, the conclusion $\eta=\overline{i}(\delta)$ is true also when $\eta=i(\delta)$. To see this, note that in this case $i(\delta)+1\subseteq X$ and therefore  $\overline{i}(\delta)=k^{-1}(i(\delta))=i(\delta)=\eta$.}.
        Lemma \ref{lemma: normal generation} 
        implies that \(F\cup \overline{i}[U]\) generates \(\overline{i}(U)\), where \(F\) denotes the tail filter on \(\eta\). 
        Since \(\tilde{U}\) is fine,
        \(F\subseteq \tilde{U}\), and 
        \(\overline i[U]=\{i(A)\cap \eta : A\in U\}\subseteq \tilde{U}\).
        Since \(F\cup \overline{i}[U]\) generates \(\overline i(U)\), it follows that \(\overline i(U)\subseteq \tilde{U}\). 

        By Lemma \ref{lemma: hull analysis}, there is some \(n \leq \omega\) such that \(\overline i = i_{0n}\) and \(k = i_{n\infty}\).
        Since \(k = i_{n\infty}\) and \(\eta \leq \crit(k)\), we have \(P(\eta)\cap \overline N = P(\eta)\cap N\), and therefore using the maximality of ultrafilters and the fact that \(\overline i(U)\subseteq \tilde{U}\), we obtain \(\overline i(U) =  \tilde{U}\). Since \(\overline i = i_{0n}\), we have
        \(\tilde{U} = \overline i(U) = i_{0n}(U)\), which proves the lemma.
        \end{proof}

\section{Some ultrafilters in the discrete Magidor extension}\label{Section: Some ultrafilters}
Throughout this section, we fix a measurable cardinal \(\kappa\), a discrete set of measurable cardinals \(\Delta\subseteq\kappa\) such that $\sup(\Delta)=\kappa$, a
sequence \(\vec U = \langle U_\delta : \delta\in \Delta\rangle\) of normal ultrafilters \(U_\delta\) on \(\delta\), and a \(V\)-generic filter \(G\) on the discrete Magidor product \(\mathbb P_\kappa\) associated with \(\vec U\).
This section is devoted to showing that  Kaplan's theorem (Corollary \ref{Cor: Kaplan's Theorem}) does not generalize to arbitrary \(\kappa\)-complete ultrafilters on \(\kappa\).

\subsection{An ultrafilter with infinitely many extensions}\label{subSection: two normals}
In this subsection, we exhibit a \(\kappa\)-complete ultrafilter \(W\) on \(\kappa\times \kappa\) that, in \(V[G]\), has infinitely many distinct extensions to \(\kappa\)-complete ultrafilters.

Before we do this, let us make some general comments on where ultrafilters in \(V[G]\) come from. Let $\tilde{W}$ be a $V[G]$-ultrafilter. Then 
\(\tilde{W}\) is uniquely determined by the following three ingredients:
\begin{itemize}
\item \(j_{\tilde{W}}^{V[G]}\restriction V\), the \textit{restricted ultrapower embedding}
\item \(j_{\tilde{W}}(G)\), the \textit{image generic}
\item \([\text{id}]_{\tilde{W}}\), the \textit{seed}
\end{itemize}
The first two ingredients determine $j_{\tilde{W}}^{V[G]}$ and given that $[\text{id}]_{\tilde{W}}=a$, we can recover $\tilde{W}$ as the ultrafilter derived from $j_{\tilde{W}}^{V[G]}$  and $a$. 

Note that an arbitrary list of ingredients need not form a recipe for cooking up a genuine $V[G]$-ultrafilter. Suppose one is given an elementary embedding \(j : V\to M\),
an \(M\)-generic filter \(H\) on \(j(\mathbb P_\kappa)\) and a point \(a\in M\).
When is there a $V[G]$-ultrafilter \(\tilde W\) whose restricted ultrapower is \(j\), image generic is \(H\), and seed \(a\)? It is easy to see this is the case if and only if the following hold:
\begin{itemize}
    \item \(j[G]\subseteq H\).
    \item \(M\subseteq \hull^{M[H]}(j[V]\cup \{H,a\})\).
\end{itemize}
We use here Corollary \ref{cor: hulls and ultrapowers} which sets up an equivalent condition for an embedding $k:V[G]\rightarrow M$ being an ultrapower embedding by some $V[G]$-ultrafilter $\tilde{W}$.

Let us turn to the description of the ultrafilter $W$ with infinitely many extensions to $V[G]$. We start with any normal ultrafilter \(D\) on \(\kappa\).
For each \(\alpha < \kappa\), let \[\delta(\alpha) = \min (\Delta\setminus \alpha)\]
and let \[W = \sum_{D} U_{\delta(\alpha)}\] 
In other words, a set \(A\subseteq \kappa\times \kappa\) belongs to \(W\) if and only if for \(D\)-almost all \(\alpha\), for \(U_{\delta(\alpha)}\)-almost all \(\beta\), \((\alpha, \beta)\in A\).
The well-known general analysis of sums of ultrafilters yields the following lemma (see for example \cite[3.5.7]{GoldbergUA}):
\begin{lemma}
    Let \(\delta = \min (j_D(\Delta)\setminus \kappa)\) and let \(U = j_D(\vec U)_\delta\). Then \(V_W = (V_D)_U\), \(j_W = j_U\circ j_D\), and \([\textnormal{id}]_W = (\kappa, \delta)\).
\end{lemma}
A $V[G]$-ultrafilter $\tilde{W}$ given by the three ingredients extends $W$ if there is an elementary embedding $k:V_W\rightarrow j_{\tilde{W}}(V)$ such that the restricted ultrapower $j^{V[G]}_{\tilde{W}}\restriction V$ is equal to $k\circ j_W$ and $k$ maps $[id]_W$ to $[id]_{\tilde{W}}$.

The first lift is the canonical extension $W^*$ (see Definition \ref{Definition: Canonical extension}), whose three ingredients are $i^W\circ j_W, \ G\times G_{\vec{s}^{\hspace{0.04 cm}W}}$ and $(\kappa,\delta)$.  Here \[i^W : V_W\to N\] is the complete iteration of \(V_W\) by \(j_W(\vec U)\restriction (\kappa,j_W(\kappa))\). (Note that the critical point of $i^W$ is above $\delta$.)

To define a non-canonical lift of $W$, we will absorb $j_U$ into the complete iteration $i^D$. Formally, let \(i^D : V_D\to N'\) be the complete iteration of \(V_D\) via
\(j_D(\vec U)\restriction (\kappa,j_D(\kappa))\).
Since \(U\) is the first ultrafilter used in the complete iteration of \(j_D(\vec U)\restriction (\kappa,j_D(\kappa))\), it is easy to see that \(N' = N\) and \[i^D=i^W\circ j_U.\]

Let $\Delta^N=i^D(\Delta)$, and
let \(\vec s^{\hspace{0.07 cm}D} = \langle s^D_\alpha : \alpha \in \Delta^N\setminus \kappa\rangle\) be the associated sets of indiscernibles.
By Proposition \ref{Prop: embedding of W*}, letting $D^*$ be the canonical extension of $D$, \(j_{D^*}^{V[G]} \restriction V = i^D\circ j_D\) and
\(j_{D^*}^{V[G]}(G) = G \times G_{\vec s^{\hspace{0.04 cm}D}}\).

Observe that \[i^W\circ j_W = i^W\circ j_U\circ j_D = i^D\circ j_D\]
Similarly, letting \(\vec s^{\hspace{0.07 cm} W} = \langle s^W_\alpha : \alpha\in\Delta^N\setminus \kappa\rangle\) denote the associated sequence of sets of indiscernibles and \(\delta^* = i^D(\delta)\), we have 
\[s^D_{\alpha} = \begin{cases}
s^W_\alpha &\text{if }\alpha\neq \delta^*\\
s^W_{\delta^*}\cup \{\delta\}&\text{if }\alpha = \delta^*\end{cases}\]
It follows that \(j_{D^*}\neq j_{W^*}\) since \[j_{W^*}(G) = G \times G_{\vec s^{\hspace{0.04 cm}W}} \neq G\times G_{\vec s^{\hspace{0.04 cm}D}} = j_{D^*}(G)\]
yet \[V[G]_{W^*} =N[G \times G_{\vec s^{\hspace{0.04 cm}W}}]=N[G \times G_{\vec s^{\hspace{0.04 cm}D}}]=V[G]_{D*}.\] 

Now let \(W'\) be the $V[G]$-ultrafilter on \(\kappa\times \kappa\) derived from \(j_{D^*}\) using \((\kappa,\delta)\). Then by Corollary \ref{cor: hulls and ultrapowers} \(j_{W'} = j_{D^*}\). To see that $W'$ lifts $W$, we note that $j_{D^*}\restriction V=i^D\circ j_W=i^W\circ j_W$ and again $i^W(\kappa,\delta)=(\kappa,\delta)$. So the three ingredients of $W'$ are $i^W\circ j_W, G\times G_{\vec s^{\hspace{0.04 cm}D}}$, and $(\kappa,\delta)$. As we noted above, $j_{D^*}(G)\neq j_{W^*}(G)$ and it follows that \(W' \neq W^*\).

Thus we have constructed two distinct \(V[G]\)-ultrafilters extending \(W\): the canonical one, and another that is Rudin-Keisler equivalent to \(D^*\).
Are these the only extensions? As the title of this subsection suggests, the answer is no and there are infinitely many more, falling into two countably infinite families: the first generalizing \(W'\) and the second generalizing \(W^*\). 
We begin with the generalizations of \(W'\), which are a bit easier to describe as they are all Rudin-Keisler equivalent to \(D^*\).
Let \(\langle \delta_n \rangle_{n < \omega}\) be the increasing enumeration of \(s^D_{\delta^*}\), the  set of indiscernibles associated with \(\delta^*\) in the complete iteration of \(V_D\).

 \begin{definition}
Let \(W^1_n\) be the $V[G]$-ultrafilter derived from
\(j_{D^*}^{V[G]}\) using \((\kappa,\delta_n)\).
\end{definition}

The reason this forms an extension of $W$ is that we can represent $i^W\circ j_W$ differently. It is a well-known fact that
$j_{U^{n+1}}=j_{j_{U^n}(U)}\circ j_{U^n}=(j_{U^n}\restriction V_U)\circ j_U$. Therefore \begin{align*}      
i^W\circ j_W&=i^D\circ j_D\\
&= i_{n+1,\theta}^D\circ j_{U^{n+1}}\circ j_D\\
&=i_{n+1,\theta}^D\circ (j_{U^n}\restriction V_U)\circ j_U\circ j_D\\
&=i_{n+1,\theta}^D\circ (j_{U^n}\restriction V_U)\circ j_W
\end{align*}

Once again it is not hard to see that Corollary \ref{cor: hulls and ultrapowers} can be applied here to conclude that for every $n<\omega$, $j^{V[G]}_{W^1_n}=j^{V[G]}_{D^*}$ and \([\text{id}]_{W^1_n} = (\kappa,\delta_n)\).
In other words, the three ingredients that determine $W^1_n$ are $i^D_{n+1,\theta}\circ j_{U^n}\restriction M_U\circ j_W$, $G\times G_{\vec{s}^{\hspace{0.03 cm}D}}$, and $(\kappa,\delta_n)$. 
Note that  $i^D_{n+1,\theta}(j_{U^n}(\kappa,\delta))=(\kappa,j_{U^n}(\delta))=(\kappa,\delta_n)$ and therefore $W^1_n$ lifts $W$. All the $W^1_n$'s are distinct as they are derived from the same normal ultrapower embedding (i.e., they are Rudin-Keisler equivalent) using different seeds.

We now turn to the second family of extensions of $W$: the generalizations of \(W^*\), which will be denoted by \(W^0_n\).
We specify the ultrafilter \(W^0_n\) by listing the three ingredients first.
As in the case of \(W^1_n\), the restricted ultrapower embedding associated with \(W^0_n\) is \(i^D\circ j_D\) and the seed is \((\kappa,\delta_n)\). The difference is in the image generic: we will have \(j_{W^0_n}(G) \neq G\times G_{\vec s^{\hspace{ 0.04 cm}D}}\). Instead,
\[j_{W^0_n}(G) = G\times G_{\vec s^{\hspace{.04cm}n}}\]
where \(\vec s^{\hspace{.07cm}n}\) is the sequence of sets of indiscernibles obtained from 
\(\vec s^{\hspace{.07cm}D}\) by removing the ordinal \(\delta_n\) from \(s^D_{\delta^*}\); that is, \(s^{n}_{\delta^*} = s^D_{\delta^*}\setminus \{\delta_n\}\) and for \(\alpha> \delta^*\),
\(s^n_\alpha = s^D_\alpha\).

It is not entirely obvious that there is an extension of $W$ to a \(V[G]\)-ultrafilter that has this restricted ultrapower embedding, image generic, and seed. To show that \(W^0_n\) exists, we need to prove that there is an elementary embedding \(\ell : V[G]\to N[G\times G_{\vec s^{\hspace{.04cm}n}}]\) extending
\(i^W\circ j_W\) such that \(\ell(G) = G\times G_{\vec s^{\hspace{.04cm}n}}\)
and \begin{equation} \label{Equation: hull}
        N[G\times G_{\vec s^{\hspace{.04cm}n}}] = \text{Hull}^{N[G\times G_{\vec s^{\hspace{.03cm}n}}]}(\ell\big[V[G]\big]\cup \{(\kappa,\delta_n)\})
\end{equation}
Then by Corollary \ref{cor: hulls and ultrapowers}, $\ell=j_{W^0_n}$ where \(W^0_n\) is the \(V[G]\)-ultrafilter on \(\kappa\times \kappa\) derived from \(\ell\) using \((\kappa,\delta_n)\).

To show that \(\ell\) exists, let \(U_n\) denote the \(n\)-th iterate \(j_{U^n}(U)\) of \(U\), and note that \(U_n\) is an \(N[G_{\vec s^{\hspace{.04cm}D}}]\)-ultrafilter on \(\delta_n\), although $U_n$ is not an element of \(N[G_{\vec s^{\hspace{.04cm}D}}]\). By L\'{e}vy-Solovay, $U_n$ generates an \(N[G\times G_{\vec s^{\hspace{.04cm}D}}]\)-ultrafilter \(U_n^*\)  and $j_{U_n^*}\restriction N[G_{\vec s^{\hspace{.04cm}D}}]=j_{U_n}^{N[G_{\vec s^{\hspace{.04cm}D}}]}$. 
We will set \begin{equation}\label{equation: ell}\ell = j_{U_n^*}\circ j_{D^*}\end{equation}

By Proposition \ref{Proposition: indiscernible after ultrapower} applied in $V_D$, \(j_{U_n}^{N[G_{\vec s^{\hspace{.04cm}D}}]}\circ i^D = i^D\),
and therefore \[\ell\restriction V=j_{U_n^*}\circ j_{D^*}\restriction V=(j_{U^*_n}\restriction N[G_{\vec s^{\hspace{.04cm}D}}])\circ (i^D\circ j_D)=i^D\circ j_D.\] Namely $\ell$ extends 
\(i^D\circ j_D\); moreover, Proposition \ref{Proposition: indiscernible after ultrapower} ensures that \(j_{U_n}(G_{\vec s}) = G_{\vec s^{\hspace{.04cm}n}}\),
and combiniting with Proposition \ref{Prop: embedding of W*} we have that \[\ell(G) = j_{U^*_n}(j_{D^*}(G))=j_{U^*_n}(G\times G_{\vec s})= G\times G_{\vec s^{\hspace{.04cm}n}}.\] 
This verifies that \(\ell\) has the correct restricted embedding and image generic.

Finally, we verify (\ref{Equation: hull}). Let \(H = \text{Hull}^{N[G\times G_{\vec s}]}(\ell\big[V[G]\big]\cup \{\kappa,\delta_n\})\).
Since \(N[G\times G_{\vec s}] = V[G]_{D^*}\), we have
\[ N[G\times G_{\vec s}] = \text{Hull}^{N[G\times G_{\vec s}]}(j_{D^*}\big[V[G]\big]\cup \{\kappa\})\]
To show \(N[G\times G_{\vec{s}}] = H\), it therefore suffices to show that
\(j_{D^*}\big[V[G]\big]\) is contained in \(H\).
For this, since \(j_{D^*}\restriction V = \ell\restriction V\), it is enough to show that \(j_{D^*}(G)\in H\). But 
\(j_{D^*}(G) = G\times G_{\vec s}\) is definable from \(\ell(G) = G\times G_{\vec s^{\hspace{.04cm}n}}\) and \(\delta_n\), as , roughly speaking,
\(G_{\vec s}= G_{\vec s^{\hspace{.04cm}n}}\cup \{\delta_n\}\).
Since \(j_{D^*}(G)\) is definable from parameters in \(H\), it belongs to \(H\).
\begin{remark}
    Note that we defined the extension \(W^0_n\) essentially by removing a single ordinal from \(\vec s^{\hspace{.04cm}D}\), to obtain  \(\vec{s}^{\hspace{.04cm}n}\).
    One might be tempted to define other similar extensions of \(W\), instead using sequences \(\vec{t}\) obtained by removing more elements of \(\vec{s}^{\hspace{.04cm}n}\). But in fact, using such a sequence \(\vec{t}\) in our specification of the three ingredients that constitute a lift of \(W\) would be fallacious, because these ingredients do not correspond to any ultrafilter in \(V[G]\). The reason is that removing any element of \(\vec{s}^{\hspace{.04cm}D}\) other than \(\delta_n\) makes Equation (\ref{Equation: hull}) above \textit{false.} 
\end{remark}
We note the following theorem, which follows from the foregoing analysis of \(W^0_1\):
\begin{theorem}\label{theorem: non-rigid ultrapower}
    In \(V[G]\), for any normal ultrafilter \(F\) on \(\kappa\),
    there is a nontrivial elementary embedding from \(V[G]_{F}\) to itself.
    \begin{proof}
        By Kaplan's theorem (Corollary \ref{Cor: Kaplan's Theorem}), \(F = D^*\) for some normal ultrafilter \(D\) of \(V\).
        As above, let \(U = j_D(\vec U)_{\delta^*}\) and let \(U^*\) be the \(V[G]_{D^*}\)-ultrafilter generated by \(U\). Then \((V[G]_{D^*})_{U^*}=V[G]_{D^*}\), and so
        \(j_U : V[G]_{D^*}\to V[G]_{D^*}\) is a nontrivial elementary embedding.
    \end{proof}
\end{theorem}
\subsection{Classifying the extensions of \(W\)}
In this section, we prove the special case of our classification of ultrafilters in \(V[G]\) 
for the ultrafilter
\(W = \sum_{D}U_{\delta(\alpha)}\). \begin{theorem}\label{Thm: Classification simple case}
     If \(\overline{W}\) is a countably complete \(V[G]\)-ultrafilter extending \(W\), then \(\overline{W} = W^i_n\) for some \(i\in \{0,1\}\) and \(n < \omega\).
\end{theorem}
This proof contains most of the key ideas of the classification and avoids some notational difficulties involved in propagating the result to arbitrary ultrafilters.
\begin{proof}
Let $\overline{W}$ be an extension of $W$ to a $V[G]$-ultrafilter.    
Since \(\overline{W}\) extends \(W\), the Rudin-Keisler projection of
\((\pi_0)_*(\overline W)\) of $\overline{W}$ using $\pi_0$ extends \(D\), where \(\pi_0 : \kappa\times \kappa\to \kappa\) denotes the projection to the first coordinate . By Kaplan's theorem (Corollary \ref{Cor: Kaplan's Theorem}), it follows that \((\pi_0)_*(\overline W)\) must be equal to \(D^*\).
Therefore there is an elementary embedding \(k : V[G]_{D^*} \to V[G]_{\overline W}\) such that \(k\circ j_{D^*} = j_{\overline W}\) and \(k(\kappa) = \kappa\). 

Note that \([\text{id}]_{\overline W} = (\kappa,\overline \delta)\) for some ordinal $\overline{\delta}>\kappa$. 
Let \(\eta\) be the least ordinal such that \(k(\eta) > \overline \delta\) and let \(\overline{U}\) denote the \(V[G]_{D^*}\)-ultrafilter on \(\eta\) derived from \(k\) using \(\overline{\delta}\). By Corollary \ref{cor: hulls and ultrapowers}, $k=j_{\overline{U}}$ and $[id]_{\overline{U}}=\overline{\delta}$ .
\begin{claim} 
For some \(n < \omega\), \(\overline \delta = \delta_n\) and
either \(\overline{U} = U_n^*\) or \(\{\delta_n\}\in \overline U\), so \(\overline U\) is principal. 
\end{claim}
Here $U_n^*$ is defined as in the paragraph preceding (\ref{equation: ell}).
Granting this claim, it is easy to see that either \(\overline W = W^0_n\) or \(\overline W = W^1_n\). 
Indeed if  $\{\delta_n\}\in\overline{U}$ then $j_{\overline{U}}$ is the identity and $j_{\overline{W}}=j_{D^*}$. Then $\overline{W}$ is derived from $j_{D^*}$ using $(\kappa,\delta_n)$ which is  by definition the ultrafilter $W^0_n$. If $\overline{U}=U^*_n$, then $\overline{W}$ is derived from $j_{U^*_n}\circ j_{D^*}$ using $(\kappa,\delta_n)$, which is by definition  $W^0_n$.   

To prove the claim, we analyze the \(N\)-ultrafilter \(\tilde{U} = \overline{U}\cap N\). We will show that for some \(n\), either \(\tilde U = U_n\) or \(\{\delta_n\}\in \tilde U\). The claim then follows since in either case,
\(\tilde U\) generates an \(N[G\times G_{\vec s}]\)-ultrafilter.

The analysis of \(\tilde U\) is an application of Lemma \ref{lemma: tilde U} in the case \(M = V_D\) and $i=i^D$, which implies that either \(\tilde{U}\) has the desired form or else \(\tilde{U} = i^D_{0\omega}(U)\). But 
the latter cannot occur, because \(\tilde{U}\) extends to a countably complete \(V[G]_{D^*}\)-ultrafilter (namely, \(\overline U\)), whereas \(i^D_{0\omega}(U)\) does not since $i^D_{0\omega}(\delta)$ has countable cofinality in $V[G]_{D^*}$.$\qedhere_{\text{Theorem \ref{Thm: Classification simple case}}}$
\end{proof}

\subsection{Some other ultrafilters}
We describe a slight variant of the classification from the previous section when the second ultrafilter in the sum does not come from the image of \(\vec U\). 
We start with any normal ultrafilter \(D\) on \(\kappa\).
For each \(\alpha < \kappa\), let \[\delta(\alpha) = \min (\Delta\setminus \alpha)\]
and let $Z_{\delta(\alpha)}\neq U_{\delta(\alpha)}$ be a normal measure on $\delta(\alpha)$.\footnote{Under UA, the existence of distinct normal measures on \(\kappa\) implies \(o(\kappa) \geq 2\), so the situation described here cannot occur unless $\delta(\alpha)$ is a limit of measurable cardinals. Since we are considering discrete sequences \(\vec U\), it is reasonable to consider the case where no element of \(\vec U\) is a limit of measurables.} 
Let \[W = \sum_{D} Z_{\delta(\alpha)}\] 

Again, one lift of \(W\) is the canonical extension $W^*$.

In this case there is no analog of the ultrafilter $W'$ considered in Section \ref{subSection: two normals} since $Z=[\alpha\mapsto Z_{\delta(\alpha)}]_D$ cannot be absorbed into $i^D$: the first ultrafilter applied in $i^D$ is $U\neq Z$. 

However, we still have the analogues of the ultrafilters \(W^0_n\).
Indeed, let us specify the ultrafilter \(W^0_n\) in this case by listing the three ingredients.
The major difference is the restricted ultrapower embedding associated with \(W^0_n\), which is no longer equal to \(i^W\circ j_W\). Instead, 
the restricted ultrapower embedding is $i^{j_{U^n}(j_W(\vec{U}))}\circ j_{U^n}\circ j_W$
and the seed is \((\kappa,\delta_n)\) where as before \(\delta_n = j_{U^n}(\delta)\). 
The image generic: we will have \(j_{W^0_n}(G) = G\times G_{\vec s^{\hspace{.04cm}n}}\)
where \(\vec s^{\hspace{.07cm}n}\) is  defined as follows: let $\vec{r}$ be the sequence of sets of indiscernibles obtained by the complete iteration of $j_{U^n}(j_W(\vec{U}))$ and let $\vec{s}^{\hspace{.07cm}n}$ be the sequence  obtained by adding $\delta_0,\dots,\delta_{n-1}$ to $\vec{r}$.

The proof that these three ingredients constitute a valid lift of \(W\) is as in the previous sections, as is the proof that the ultrafilters \(W^0_n\) exhaust all lifts of \(W\) to \(V[G]\). 
\section{Classification of ultrafilters in the Magidor extension}\label{Section: Classification}
Let \(\Delta\) be a discrete set of measurable cardinals with supremum \(\kappa\), and let \(\vec U = \langle U_\alpha : \alpha\in \Delta\rangle\) be a sequence of normal measures \(U_\alpha\) on \(\alpha\). 
Our goal in this section is to classify the extensions of sums of normal ultrafilters to \(V[G]\) where \(G\subseteq \mathbb P_\kappa\) is \(V\)-generic for the Magidor iteration of \(\vec U\). We begin with Section \ref{subsection: someextensions} by defining a family of extensions of a given sum of normal ultrafilters and analyzing the relationship between their ultrapowers. Then in Section \ref{subsection: classification}, we show that this family exhausts all extensions of the sum. 
\subsection{Extensions of an iterated sum of normal ultrafilters}\label{subsection: someextensions}

\begin{definition} Let $M$ be a transitive model of set theory. 
    \begin{itemize}
        \item A \textit{finite iteration of $M$} is a sequence \((D_m : m < n)\) such that for each \(m < n\), \(D_m\) is a \(M_{D_0,\cdots, D_{m-1}}\)-ultrafilter.
        \item Let $M_m = M_{D_0,...,D_{m-1}}$ and $j_{m_0m_1}:M_{m_0}\to M_{m_1}$ be the iterated ultrapower embedding $j_{D_{m_0}...D_{m_1-1}}$.
        \item An iteration is \textit{internal} if  for each \(m < n\),
        \(D_m\in M_{m}\). 
        \item An iteration is \textit{normal} if for each $m<n$, $D_m$ is a $M_m$-normal ultrafilter on $\delta_m$ and $\delta_0<\delta_1<...<\delta_{n-1}$.
        \end{itemize}
\end{definition}
We say that an iteration $(D_0,...,D_{n-1})$ is \textit{below} an ordinal $\gamma$ if for each $m<n$, $D_m$ is an ultrafilter on $\delta_m < j_{D_0,...,D_{m-1}}(\gamma)$.

\begin{definition}
The \textit{sum} of a normal iteration \((D_0,\dots,D_{n-1})\), denoted \(\Sigma(D_0,\dots,D_{n-1})\), is the unique ultrafilter \(W\) on \(\kappa^n\) 
such that \(j_W = j_{D_0,...,D_{n-1}}\)
and \([\text{id}]_{W} = (\delta_0,\delta_1,\dots,\delta_{n-1})\), where \(\kappa\) is least such that \((D_0,\dots,D_{n-1})\) is below \(\kappa\).
\end{definition}

%In the cases we consider, we will have that \(\kappa\) is measurable
%and \(\kappa = \delta_0 < \cdots < \delta_{n-1} \leq j_{0n}(\kappa)\),
%so that \(\Sigma(D_0,\dots,D_{n-1})\) is a \(\kappa\)-complete ultrafilter on \(\kappa^n\).

The following theorem \cite[Theorems 5.3.8 and 5.3.13]{GoldbergUA} explains why it is natural to consider such sums in our classification:
\begin{theorem}[UA]\label{Them: sums of numral under UA}
    Assume that there is no cardinal $\kappa$ with $o(\kappa)=2^{2^\kappa}$.
    Then every countably complete ultrafilter is Rudin-Keisler equivalent to the sum of a normal iteration.
\end{theorem}
Fix an internal iteration $(D_0,...,D_{n-1})$ of normal ultrafilters. Our plan is to classify all extensions of \(\Sigma(D_0,\dots, D_{n-1})\) to \(V[G]\) where \(G\subseteq \mathbb P_\kappa\) is \(V\)-generic. The first step is to define a countable family
\(\{\Sigma(D_0,\dots, D_{n-1})^u_x\}_{u,x}\)
of extensions of \(\Sigma(D_0,\dots, D_{n-1})\),
associated to some finite parameters \(u,x\).
 
Let \(d\) be the set of \(m \leq n\) such that \(\delta_m\in j_{0m}(\Delta)\), and let \(d'\) be the set of \(m\in d\) such that
\(D_m = j_{0m}(\vec U)_{\delta_m}\).

We will associate to each \(u : d'\to \{0,1\}\) and \(x : d\to \omega\) an extension \(\Sigma(D_0,\dots,D_{n-1})^u_x\) of \(\Sigma(D_0,\dots,D_{n-1})\).
Let us briefly explain the role of these parameters in terms of the three ingredients that determine the extension \(\overline{W} = \Sigma(D_0,\dots,D_{n-1})^u_x\). (The ``three ingredients" framework is introduced Section \ref{subSection: two normals}.)
Note that the
seed \([id]_{\overline{W}}\) has the form \(( \tilde{\delta}_0,\dots,\tilde{\delta}_{n-1})\).
The number \(u(\delta_m)\in \{0,1\}\) will determine whether or not \(\tilde{\delta}_m\) is a Pr\'ikry point in the image generic \(j_{\overline{W}}(G)\).
The number \(x(\delta_m) < \omega\) determines how many Pr\'ikry points are below \(\tilde{\delta}_m\) in the first Pr\'ikry sequence above \(\tilde{\delta}_m\).

By recursion on $m\leq n$, we define a \(V[G]\)-ultrafilter \(\overline{W}_m=\Sigma(D_0,\dots,D_{m-1})^u_x\) extending \(W_m=\Sigma(D_0,\dots,D_{m-1})\). We will also define an external iteration $(E_0,E_1,...,E_{m-1})$ of $V_{W_m}$ below $\delta_m$ whose well-founded last model $P_m$ completely iterates into $j_{\overline{W}_m}^{V[G]}(V)$. More precisely, 
\begin{itemize}
    \item Let $e_m:V_{W_m}\rightarrow P_m$ be the iterated ultrapower by $(E_0,...,E_{m-1})$.
    \item Let $i^{P_m}:P_m\rightarrow N_m$ be the complete iteration of $P_m$ by $e_m(j_{W_m}(\vec{U}))$ above $\kappa$.
\end{itemize}
    We will maintain that the following hold:
    \begin{itemize}
        \item $N_m=j_{\overline{W}_m}(V)$
        \item \(j_{\overline{W}_m}\restriction V=i^{P_m}\circ e_m\circ j_{W_m}\)
        \item \(i^{P_m}\circ e_m([id]_{W_m})=[id]_{\overline{W}_m}\)
    \end{itemize}
    This ensures, in particular, that the ultrafilter $\overline{W}_m$ lifts $W_m$.

Fix  \(m < n\) and assume that we have already defined 
\(\overline{W}_m\) and the associated external iteration $(E_0,...,E_{m-1})$ satisfying the bullets above.
We will define \(\overline{W}_{m+1}\) and $E_{m}$ and show in Claim \ref{claim: maintain} that our recursive hypotheses are maintained.

Let $i:V_{W_m}\rightarrow N_m$ be the composition $i^{P_m}\circ e_m$.
We will define an \(N_m\)-ultrafilter \(\tilde{D}_m\) extending 
\[\{i(A)\cap \eta : A\in D_m\}\]
for some \(\eta \leq i(\delta_m)\). 
The ultrafilter \(\tilde{D}_m\) will generate a \(V[G]_{\overline W_m}\)-ultrafilter, which we denote by \(D_m^{u,x}\), and we will set
\(\overline{W}_{m+1}= \Sigma(\overline W_m,D_m^{u,x})\).
Therefore in the end we will have \(\overline{W}_{n} = \Sigma(D_0^{u,x},\dots,D_{n-1}^{u,x})\).

If \(m \notin d\), \(\tilde{D}_m = i(D_m)\) and $E_m$ is the principal $P_m$-ultrafilter concentrated at $e_m(\delta_m)$. 

Now suppose \(m \in d\), in which case \(\tilde{D}_m\) will depend on \(u\) and \(x\). Define $E_m$ to be the external ultrafilter $e_m(j_{W_m}(\vec{U})_{\delta_m})^{x(m)}$. 
For each \(\beta\) less than the length of the complete iteration of \(P_m\), let \(\delta^\beta_m = i_{0\beta}^{P_m}(e_m(\delta_m))\) and \(D^\beta_m = i_{0\beta}^{P_m}(e_m(D_m))\).

Let \(\alpha\) be the first stage of the complete iteration of \(P_m\) such that \(\crit(i_{\alpha\alpha+1}^{P_m}) = \delta^\alpha_m\). 
If $m\in d'$ and \(u(m)=1\), let \(\tilde{D}_m\) be the principal ultrafilter concentrated at \(\delta^{\alpha+x(m)}_m\). Otherwise, let \(\tilde{D}_m = D^{\alpha+x(m)}_m\). 

We have
\(\{i(A)\cap \eta : A\in D_m\}\subseteq \tilde{D}_m\)
where     $$\eta =\begin{cases} i(\delta_m) & \text{if } m\notin d\\
    \delta^{\alpha+x(m)}_m+1 & m\in d' \text{ and } u(m)=1\\  \delta^{\alpha+x(m)}_m & \text{otherwise}
    \end{cases}$$
\begin{claim}
     \(\tilde{D}_m\) generates a \(V[G]_{\overline W_m}\)-ultrafilter \(D^{u,x}_m\)
\end{claim}
\begin{proof}
    The proof is by cases. In the first case when $m\notin d$ we appeal to Theorem \ref{Thm: Delta-mild generates an ultraiflter}. In the second case, when $m\in d'$ and $u(m)=1$, $\tilde{D}_m$ is principal and therefore trivially generates a $V[G]_{\overline{W}_m}$-ultrafilter. In the last case, \(\tilde{D}_m\) is a \(\gamma\)-complete \(N_m\)-ultrafilter on \(\gamma = \delta^{\alpha+x(m)}_m\), so it generates an ultrafilter by L\'{e}vy-Solovay. 
\end{proof}
Finally, define \[\overline{W}_{m+1} = \Sigma(\overline W_m,D_m^{u,x}).\]

 To complete the induction, we must prove the following claim:
\begin{claim}\label{claim: maintain}
    $N_{m+1}=j_{\overline{W}_{m+1}}(V)$, $  j_{\overline{W}_{m+1}}\restriction V=i^{P_{m+1}}\circ e_{m+1}\circ j_{W_{m+1}}$, and $
    i^{P_{m+1}}\circ e_{m+1}([id]_{W_{m+1}})=[id]_{\overline{W}_{m+1}}.$

\end{claim}

\begin{proof}
 We consider the three cases. 
    For the first case, assume \(m\notin d\). Let $G_m$ be the $N_m$-generic filter given by the sequence of sets of indiscernibles associated with the complete iteration $i^{P_m}$. By the induction hypothesis, $V[G]_{\overline{W}_m}=N_m[G\times G_m]$. 
    By definition, in this case, \(\tilde{D}_m = i(D_m)\). Let $D^*_m$ be the $N_m[G_m]$-ultrafilter generated by $\tilde{D}_m$.  
\begin{figure}[H]
    \centering   % https://q.uiver.app/#q=WzAsMTMsWzAsNSwiUF9tIl0sWzEsNSwiUF97bSsxfSJdLFswLDQsIk5fe20sXFxhbHBoYX0iXSxbMSw0LCJOX3ttKzEsXFxhbHBoYX0iXSxbMSwyLCJOX3ttKzF9Il0sWzAsMCwiVltHXV97XFxiYXJ7V31fe219fSJdLFswLDIsIk5fbSJdLFswLDEsIk5fbVtHX21dIl0sWzEsMSwiTl97bSsxfVtHX3ttKzF9XSJdLFsxLDAsIlZbR11fe1xcYmFye1d9X3ttKzF9fSJdLFsxLDMsIk5fe20rMSxqX3tEXjBfbX0oXFxhbHBoYSl9Il0sWzAsNiwiVl97V19tfSJdLFsxLDYsIlZfe1dfe20rMX19Il0sWzYsNCwial97RF4wX219XFxyZXN0cmljdGlvbiBOX20iXSxbNiw3LCIiLDAseyJzdHlsZSI6eyJib2R5Ijp7Im5hbWUiOiJkYXNoZWQifSwiaGVhZCI6eyJuYW1lIjoibm9uZSJ9fX1dLFs0LDgsIiIsMix7InN0eWxlIjp7ImJvZHkiOnsibmFtZSI6ImRhc2hlZCJ9LCJoZWFkIjp7Im5hbWUiOiJub25lIn19fV0sWzUsOSwiRF57dSx4fV9tIl0sWzcsNSwiIiwwLHsic3R5bGUiOnsiYm9keSI6eyJuYW1lIjoiZGFzaGVkIn0sImhlYWQiOnsibmFtZSI6Im5vbmUifX19XSxbOCw5LCIiLDIseyJzdHlsZSI6eyJib2R5Ijp7Im5hbWUiOiJkYXNoZWQifSwiaGVhZCI6eyJuYW1lIjoibm9uZSJ9fX1dLFs3LDgsIkReKl9tIl0sWzExLDEyLCJEX20iLDJdLFsxMSwwLCJlX20iXSxbMTIsMSwiZV97bSsxfSIsMl0sWzAsMiwiaV57UF9tfV97MCxcXGFscGhhfSJdLFsxLDMsImlee1Bfe20rMX19X3swLFxcYWxwaGF9IiwyXSxbMCwxLCJEXjBfbSJdLFsyLDMsIkReXFxhbHBoYV9tIl0sWzIsMTAsImpfe0ReMF9tfVxccmVzdHJpY3Rpb24gTl97bSxcXGFscGhhfSJdLFsyLDYsImlee1BfbX1fe1xcYWxwaGEsXFxpbmZ0eX0iXSxbMywxMCwiaV57UF97bSsxfX1fe1xcYWxwaGEsal97RF4wX219KFxcYWxwaGEpfSIsMl0sWzEwLDQsImlee1Bfe20rMX19X3tqX3tEXjBfbX0oXFxhbHBoYSksXFxpbmZ0eX0iLDJdXQ==
\[\begin{tikzcd}
	{V[G]_{\overline{W}_{m}}} & {V[G]_{\overline{W}_{m+1}}} \\
	{N_m[G_m]} & {N_{m+1}[G_{m+1}]} \\
	{N_m} & {N_{m+1}} \\
	& {N_{m+1,j_{D^0_m}(\alpha)}} \\
	{N_{m,\alpha}} & {N_{m+1,\alpha}} \\
	{P_m} & {P_{m+1}} \\
	{V_{W_m}} & {V_{W_{m+1}}}
	\arrow["{D^{u,x}_m}", from=1-1, to=1-2]
	\arrow[phantom, sloped,"\subseteq", from=2-1, to=1-1]
	\arrow["{D^*_m}", from=2-1, to=2-2]
	\arrow[phantom, sloped,"\subseteq", from=2-2, to=1-2]
	\arrow[phantom, sloped,"\subseteq", from=3-1, to=2-1]
	\arrow["{j_{D^0_m}\restriction N_m}", from=3-1, to=3-2]
	\arrow[phantom, sloped,"\subseteq", from=3-2, to=2-2]
	\arrow["{i^{P_{m+1}}_{j_{D^0_m}(\alpha),\infty}}"', from=4-2, to=3-2]
	\arrow["{i^{P_m}_{\alpha,\infty}}", from=5-1, to=3-1]
	\arrow["{j_{D^0_m}\restriction N_{m,\alpha}}", from=5-1, to=4-2]
	\arrow["{D^\alpha_m}", from=5-1, to=5-2]
	\arrow["{i^{P_{m+1}}_{\alpha,j_{D^0_m}(\alpha)}}"', from=5-2, to=4-2]
	\arrow["{i^{P_m}_{0,\alpha}}", from=6-1, to=5-1]
	\arrow["{D^0_m}", from=6-1, to=6-2]
	\arrow["{i^{P_{m+1}}_{0,\alpha}}"', from=6-2, to=5-2]
	\arrow["{e_m}", from=7-1, to=6-1]
	\arrow["{D_m}"', from=7-1, to=7-2]
	\arrow["{e_{m+1}}"', from=7-2, to=6-2]
\end{tikzcd}\]

    \caption{The case that $m\notin d$}
    \label{fig:enter-label}
\end{figure}
 Note that by L\'{e}vy-Solovay,   \[j_{D^{u,x}_m}\restriction {N_m[G_{m}]} = j_{D^*_m}^{N_m[G_{m}]}\]
     The key point is that
     \begin{equation}\label{Equation key} j^{N_m[G_m]}_{D^*_m}\restriction N_m=j_{D^0_m}\restriction N_m\end{equation}
     To see this, let $\alpha$ be the least ordinal such that $\crit(i^{P_m}_{\alpha,\alpha+1})>e_m(\delta_m)$ and consider the model $N_{m,\alpha}[G_m\restriction \delta^0_m]$.  Note that $G_m\restriction \delta^0_m$ is $N_{\alpha,m}$-generic for a forcing which has smaller cardinality than the critical point of the embedding $i_{\alpha,\infty}^{P_m}$. So by the L\'evy-Solovay argument, $i^{P_m}_{\alpha,\infty}$ lifts to an embedding $i^*_{\alpha,\infty}:N_{m,\alpha}[G_m\restriction \delta^0_m]\to N_m[G_m\restriction \delta^0_m]$. 
     It is easy to see that $i^*_{\alpha,\infty}$ 
    is the complete iteration of 
     \(N_{m,\alpha}[G_m\restriction \delta^0_m]\) via the canonical lift of the sequence \(i^{P_m}_{0\alpha}(e_m(\vec U))\restriction (\delta^0_m,\infty)\).
     Moreover \(N_m[G_m]\) is the generic extension of the final model of this iteration by the filter obtained from the associated sequence of sets of indiscernibles. 
    Therefore we can apply Lemma \ref{Lemma: Generic} in \(N_{m,\alpha}[G_m\restriction \delta^0_m]\) to conclude that $N_m[G_m]$ is closed under $\delta^0_m$-sequences from $N_{m,\alpha}[G_m\restriction \delta^0_m]$.
     
     Note that \begin{equation}\label{equation: restriction is commutative}
     j_{D^\alpha_m}\circ i^{P_m}_{0,\alpha}=i^{P_{m+1}}_{0,\alpha}\circ j_{D^0_m}
     .\end{equation}
     This follows once we prove that $i^{P_{m+1}}_{0,\alpha}=i^{P_m}_{0,\alpha}\restriction P_{m+1}$. This is a routine induction on $\beta\leq\alpha$ using that $\crit(j_{D^0_\alpha})$ is greater than all the measurable cardinals appearing in the iteration $i^{P_m}_{0,\alpha}$.

     By Mitchell's lemma (Lemma \ref{Lemma: Mitchell}) applied in $P_m$ with $W=D^0_m$, \[j^{N_{m,\alpha}[G_m\restriction \delta^0_m]}_{D^*_m}=j_{D^0_m}\restriction N_{m,\alpha}[G_m\restriction \delta^0_m].\]
     Since $N_m[G_m]$ is closed under $\delta^0_m$-sequences from $N_{m,\alpha}[G_m\restriction \delta^0_m]$,
     this implies that
     \begin{equation}\label{equation: closure under sequences} j^{N_{m}[G_m]}_{D^*_m}=
     j^{N_{m,\alpha}[G_m\restriction \delta^0_m]}_{D^*_m}\restriction N_m[G_m] = j_{D^0_m}\restriction N_m[G_m]\end{equation}
     which proves (\ref{Equation key}).

    By definition of $e_{m+1}$, \(e_{m+1}= e_m\restriction V_{W_{m+1}}\) ensuring that
    $$j_{D_m^0}\circ e_m=e_{m+1}\circ j_{D_m}.$$
    Also, \(i^{P_{m+1}} = j_{D^0_m}(i^{P_m})\).
    We have \[j_{D_m^{u,x}}
    \circ i^{P_m} = j_{D^*_m}^{N[G_{m}]}\circ i^{P_m} = j_{D^0_m}\circ i^{P_m} = i^{P_{m+1}} \circ j_{D^0_m}^{P_m}\]
    Moreover
    \[j^{P_m}_{D^0_m}\circ e_m = e_{m+1}\circ j_{D_m}\]
    Combining these equations, we get
    \begin{align*}
        i^{P_{m+1}}\circ e_{m+1}\circ j_W &= i^{P_{m+1}} \circ e_{m+1} \circ j_{D_m}\circ j_{W_m} \\
        &= i^{P_{m+1}}\circ j^{P_m}_{D^0_m}\circ e_m\circ j_{W_m} \\
        &= j_{D_m^{u,x}}\circ i^{P_m} \circ e_m \circ j_{W_m} \\
        &= j_{D_m^{u,x}}\circ j_{\overline W_m}\restriction V \\
        &= j_{\overline W_{m+1}}\restriction V
    \end{align*}
    To finish the case $m\notin d$, by the normality of $D^{u,x}_m$,  $[id]_{D^{u,x}_m}$ is the ordinal $\beta$ over which $D^{u,x}_m$ is an ultrafilter. On the other hand, $i^{P_{m+1}}(e_{m+1}(\delta_m))=i^{P_{m+1}}_{0,\alpha}(e_{m}(\delta_m))=\beta$. 
    Hence $i^{P_{m+1}}(e_{m+1}([id]_{W_{m+1}})=[id]_{\overline{W}_{m+1}}$.
    
    Next consider the case where $m\in d'$ and $u(m)=1$. By definition $\tilde{D}_m$ is $p_{\delta_m^{\alpha+x(m)}}$, the principal ultrafilter concentrated at $\delta_m^{\alpha+x(m)}$. The following diagram commutes:

\begin{figure}[H]
    \centering
\[\begin{tikzcd}
	{V[G]_{\overline{W}_{m}}} & {V[G]_{\overline{W}_{m+1}}} \\
	{N_m[G_m]} & {N_{m+1}[G_{m+1}]} \\
	{N_m} & {N_{m+1}} \\
	{N_{m,\alpha+x(m)+2}} \\
	{N_{m,\alpha+x(m)+1}} & {N_{m+1,\alpha+1}} \\
	{N_{m,\alpha+x(m)}} & {N_{m+1,\alpha}} \\
	{N_{m,\alpha}} & {N'_{m,\alpha}} \\
	{P_m} & {P'_m} & {P_{m+1}} \\
	{V_{W_m}} & {V_{W_{m+1}}}
	\arrow["{p_{\delta^{\alpha+x(m)}_m}}"', tail reversed, from=1-1, to=1-2]
	\arrow[phantom, sloped,"\subseteq", from=2-1, to=1-1]
	\arrow["{p_{\delta^{\alpha+x(m)}_m}}"', tail reversed, from=2-1, to=2-2]
	\arrow[phantom, sloped,"\subseteq", from=3-1, to=2-1]
	\arrow[phantom, sloped,"\subseteq", from=2-2, to=1-2]
	\arrow["{p_{\delta^{\alpha+x(m)}_m}}"', tail reversed, from=3-1, to=3-2]
	\arrow[phantom, sloped,"\subseteq", from=3-2, to=2-2]
	\arrow["{i_{\alpha+x(m)+2,\infty}^{P_m}}"', from=4-1, to=3-1]
	\arrow["{p_{\delta^{\alpha+x(m)}_m}}", tail reversed, from=4-1, to=5-2]
	\arrow[from=5-1, to=4-1]
	\arrow["{i^{P_{m+1}}_{\alpha+1,\infty}}"', from=5-2, to=3-2]
	\arrow["{D^{\alpha+x(m)}_m}", from=6-1, to=5-1]
	\arrow["{D^{\alpha+x(m)}_m}", from=6-1, to=6-2]
	\arrow["{p_{\delta^{\alpha+x(m)}_m}}"', tail reversed, from=6-2, to=5-1]
	\arrow[from=6-2, to=5-2]
	\arrow["{i_{0,\alpha}^{P_m}(E_m)}", from=7-1, to=6-1]
	\arrow["{D^\alpha_m}", from=7-1, to=7-2]
	\arrow["{i_{0,\alpha}^{P_m}(E_m)}", from=7-2, to=6-2]
	\arrow["{i^{P_m}}", curve={height=-70pt}, from=8-1, to=3-1]
	\arrow["{i^{P_m}_{0,\alpha}}", from=8-1, to=7-1]
	\arrow["{D^0_m}", from=8-1, to=8-2]
	\arrow["{i^{P'_{m}}_{0,\alpha}}", from=8-2, to=7-2]
	\arrow["{E_m}", from=8-2, to=8-3]
	\arrow["{i^{P_{m+1}}}"', curve={height=25pt}, from=8-3, to=3-2]
	\arrow["{i^{P_{m+1}}_{0,\alpha}}"', from=8-3, to=6-2]
	\arrow["{e_m}", from=9-1, to=8-1]
	\arrow["{D_m}"', from=9-1, to=9-2]
	\arrow["{e_m}"', from=9-2, to=8-2]
	\arrow["{e_{m+1}}"', from=9-2, to=8-3]
\end{tikzcd}\]
 \caption{The case that $u(m)=1$.}
    \label{fig: principal case}
\end{figure}
The commutativity of the second square from the bottom is proved as in Equation \ref{equation: restriction is commutative}.
The only other part of the diagram whose commutativity is not immediate is  $$i^{P_{m+1}}_{0,\alpha}\circ j_{E_m}=j_{i^{P_m}_{0,\alpha}(E_m)}\circ i^{P'_m}_{0,\alpha}$$
where $P'_m=e_m(V_{W_{m+1}})$ and $i^{P'_m}_{0,\alpha}:P'_m\rightarrow N'_{m,\alpha}$ is the $\alpha^{\text{th}}$ stage of the complete iteration of $P'_m$ by $e_m(j_{W_{m+1}}(\vec{U}))$ above $\kappa$.
% https://q.uiver.app/#q=WzAsNCxbMCwxLCJQJ19tIl0sWzEsMSwiUF97bSsxfSJdLFsxLDAsIk5fe20rMX0sXFxhbHBoYSJdLFswLDAsIk4nX3ttLFxcYWxwaGF9Il0sWzAsMSwiRV9tIl0sWzEsMiwiaV57UF97bSsxfX1fezAsXFxhbHBoYX0iLDJdLFswLDMsImlee1AnX219X3swLFxcYWxwaGF9Il0sWzMsMiwiaV57UF9tfV97MCxcXGFscGhhfShFX20pIl1d
\[\begin{tikzcd}
	{N'_{m,\alpha}} & {N_{m+1,\alpha}} \\
	{P'_m} & {P_{m+1}}
	\arrow["{i^{P_m}_{0,\alpha}(E_m)}", from=1-1, to=1-2]
	\arrow["{i^{P'_m}_{0,\alpha}}", from=2-1, to=1-1]
	\arrow["{E_m}", from=2-1, to=2-2]
	\arrow["{i^{P_{m+1}}_{0,\alpha}}"', from=2-2, to=1-2]
\end{tikzcd}\]
This commutativity is true since $$ j_{i^{P_m}_{0,\alpha}(E_m)}^{N'_{m,\alpha}} = j_{E_m}\restriction N'_{m,\alpha}\text{ and }i_{0,\alpha}^{P_{m+1}}=j_{E_m}(i_{0,\alpha}^{P'_m}).$$
We include some details on how to show
\(j_{i^{P_m}_{0,\alpha}(E_m)}^{N_{m,\alpha}'} = j_{E_m}\restriction N'_{m,\alpha}\).
For this, we use that 
 \(j_{i^{P_m}_{0,\alpha}(E_m)}^{N'_{m,\alpha}} = j_{i^{P_m}_{0,\alpha}(E_m)}^{N_{m,\alpha}}\restriction N_{m,\alpha}'\).
 Since \(E_m\in P_m\),
 we can use Kunen's commuting ultrapowers lemma, as it appears in Woodin's
 \cite[Lemma 3.30]{Woodin},
 to conclude that \[j_{i^{P_m}_{0,\alpha}(E_m)}^{N_{m,\alpha}} = j_{E_m}^{P_m}\restriction N_{m,\alpha}.\]
    (We apply Woodin's lemma in \(P_m\) with \(j = i_{0,\alpha}^{P_m}\) and \(E = E_m\). Note that some of the generality of Woodin's lemma is not necessary here since \(j\) is definable over \(P_m\) rather than generic.)
    
    Since $j_{\overline{W}_{m+1}}=j_{\overline{W}_m}$ then the commutativity of the diagram in Figure \ref{fig: principal case} can be used to deduce that $$j_{\overline{W}_{m+1}}\restriction V=j_{\overline{W}_m}\restriction V=i^{P_{m}}\circ e_{m}\circ j_{W_{m}}=i^{P_{m+1}}\circ e_{m+1}\circ j_{W_{m+1}}.$$
In particular, 
$$N_{m+1}=N_m=j_{\overline{W}_m}(V)=j_{\overline{W}_{m+1}}(V)$$
Finally, we show that $i^{P_{m+1}}\circ e_{m+1}([id]_{W_{m+1}})=[id]_{\overline{W}_{m+1}}$. By the induction hypothesis it suffices to show that $i^{P_{m+1}}(e_{m+1}(\delta_m))=\delta^{\alpha+x(m)}_m$:
$$i^{P_{m+1}}(e_{m+1}(\delta_m))=i^{P_{m+1}}_{\alpha,\infty}(i^{P_{m+1}}_{0,\alpha}(e_{m+1}(\delta_m)))=i^{P_{m+1}}_{\alpha,\infty}(j_{i^{P_m}_{0,\alpha}(E_m)}(i_{0,\alpha}^{P'_m}(e_m(\delta_m))=$$
$$=i^{P_{m+1}}_{\alpha,\infty}(j_{i^{P_m}_{0,\alpha}(E_m)}(i_{0,\alpha}^{P_m}(e_m(\delta_m))=i^{P_{m+1}}_{\alpha,\infty}(\delta^{\alpha+x(m)}_m)=\delta^{\alpha+x(m)}_m.$$

Finally consider the case where where either $m\in d\setminus d'$ or $m\in d'$ but $u(m)=0$.  Let $\ell$ be the restriction to $N_m$ of the ultrapower embedding of $N_{m,\alpha+x(m)}$ by $D^{\alpha+x(m)}_{m}$. 
\begin{figure}
    \centering
% https://q.uiver.app/#q=WzAsMTUsWzAsNywiVl97V19tfSJdLFsxLDcsIlZfe1dfe20rMX19Il0sWzAsNiwiUF9tIl0sWzEsNiwiUCdfbSJdLFsyLDYsIlBfe20rMX0iXSxbMCw1LCJOX3ttLFxcYWxwaGF9Il0sWzEsNSwiTidfe20sXFxhbHBoYX0iXSxbMCw0LCJOX3ttLFxcYWxwaGEreChtKX0iXSxbMSw0LCJOX3ttKzEsXFxhbHBoYX0iXSxbMCwyLCJOX20iXSxbMCwxLCJOX21bR19tXSJdLFswLDAsIlZbR11fe1xcYmFye1d9X3ttfX0iXSxbMSwyLCJOX3ttKzF9Il0sWzEsMSwiTl97bSsxfVtHX3ttKzF9XSJdLFsxLDAsIlZbR11fe1xcYmFye1d9X3ttKzF9fSJdLFswLDEsIkRfbSJdLFswLDIsImVfbSJdLFsxLDMsImVfbSIsMl0sWzEsNCwiZV97bSsxfSIsMl0sWzMsNCwiRV9tIl0sWzIsMywiRF4wX20iXSxbMiw1LCJpXntQX219X3swLFxcYWxwaGF9Il0sWzMsNiwiaV57UCdfe219fV97MCxcXGFscGhhfSJdLFs1LDYsIkReXFxhbHBoYV9tIl0sWzUsNywiaV57UF9tfV97MCxcXGFscGhhfShFX20pIl0sWzYsOCwiaV57UF9tfV97MCxcXGFscGhhfShFX20pIl0sWzQsOCwiaV57UF97bSsxfX1fezAsXFxhbHBoYX0iLDJdLFs3LDgsIkRee1xcYWxwaGEreChtKX1fbSJdLFs3LDksImlfe1xcYWxwaGEreChtKSxcXGluZnR5fV57UF9tfSJdLFs5LDEwLCIiLDIseyJzdHlsZSI6eyJib2R5Ijp7Im5hbWUiOiJkYXNoZWQifSwiaGVhZCI6eyJuYW1lIjoibm9uZSJ9fX1dLFsxMCwxMSwiIiwyLHsic3R5bGUiOnsiYm9keSI6eyJuYW1lIjoiZGFzaGVkIn0sImhlYWQiOnsibmFtZSI6Im5vbmUifX19XSxbOCwxMiwiaV57UF97bSsxfX1fe1xcYWxwaGEsXFxpbmZ0eX0iLDJdLFs5LDEyLCJcXGVsbCJdLFsxMiwxMywiIiwwLHsic3R5bGUiOnsiYm9keSI6eyJuYW1lIjoiZGFzaGVkIn0sImhlYWQiOnsibmFtZSI6Im5vbmUifX19XSxbMTMsMTQsIiIsMCx7InN0eWxlIjp7ImJvZHkiOnsibmFtZSI6ImRhc2hlZCJ9LCJoZWFkIjp7Im5hbWUiOiJub25lIn19fV0sWzExLDE0LCJEXnt1LHh9X20iXSxbMTAsMTMsIkReKiJdXQ==
\[\begin{tikzcd}
	{V[G]_{\overline{W}_{m}}} & {V[G]_{\overline{W}_{m+1}}} \\
	{N_m[G_m]} & {N_{m+1}[G_{m+1}]} \\
	{N_m} & {N_{m+1}} \\
	\\
	{N_{m,\alpha+x(m)}} & {N_{m+1,\alpha}} \\
	{N_{m,\alpha}} & {N'_{m,\alpha}} \\
	{P_m} & {P'_m} & {P_{m+1}} \\
	{V_{W_m}} & {V_{W_{m+1}}}
	\arrow["{D^{u,x}_m}", from=1-1, to=1-2]
	\arrow[phantom, sloped,"\subseteq", from=2-1, to=1-1]
	\arrow["{D^*}", from=2-1, to=2-2]
	\arrow[phantom, sloped,"\subseteq", from=2-2, to=1-2]
	\arrow[phantom, sloped,"\subseteq", from=3-1, to=2-1]
	\arrow["\ell", from=3-1, to=3-2]
	\arrow[phantom, sloped,"\subseteq", from=3-2, to=2-2]
	\arrow["{i_{\alpha+x(m),\infty}^{P_m}}", from=5-1, to=3-1]
	\arrow["{D^{\alpha+x(m)}_m}", from=5-1, to=5-2]
	\arrow["{i^{P_{m+1}}_{\alpha,\infty}}"', from=5-2, to=3-2]
	\arrow["{i^{P_m}_{0,\alpha}(E_m)}", from=6-1, to=5-1]
	\arrow["{D^\alpha_m}", from=6-1, to=6-2]
	\arrow["{i^{P_m}_{0,\alpha}(E_m)}", from=6-2, to=5-2]
	\arrow["{i^{P_m}_{0,\alpha}}", from=7-1, to=6-1]
	\arrow["{D^0_m}", from=7-1, to=7-2]
	\arrow["{i^{P'_{m}}_{0,\alpha}}", from=7-2, to=6-2]
	\arrow["{E_m}", from=7-2, to=7-3]
	\arrow["{i^{P_{m+1}}_{0,\alpha}}"', from=7-3, to=5-2]
	\arrow["{e_m}", from=8-1, to=7-1]
	\arrow["{D_m}", from=8-1, to=8-2]
	\arrow["{e_m}"', from=8-2, to=7-2]
	\arrow["{e_{m+1}}"', from=8-2, to=7-3]
\end{tikzcd}\]    \caption{The case where $u(m)\neq 1$}
    \label{fig: the case otherwise}
\end{figure}
   Note that the bottom part of the diagram in Figure \ref{fig: the case otherwise} is identical to the bottom part of the diagram in Figure \ref{fig: principal case} and in particular it commutes. In fact the whole diagram commutes and the key to that is that the embedding $\ell$ is the restriction to $N_m$ of $j^{N_m[G_m]}_{D^*}$ where $D^*$ is the $N_m[G_m]$-ultrafilter generated by  $D^{\alpha+x(m)}_m$. 
   The justification for this is as in Equation \ref{equation: closure under sequences}.  
   The commutativity of the rest of the diagram is a straightforward verification, and the remainder of the proof of the claim in this case is then identical to the previous part. 
   
   This completes the proof, but let us note here that in the case where $m\in d'$ we obtain
       $$N_{m+1}=j_{\overline{W}_{m+1}}(V)=j_{\overline{W}_m}(V)=N_m.$$
   This is because in this case $N_{m+1,\alpha}=N_{m,\alpha+x(m)+1}$. Moreover, \begin{equation}\label{N_m=N_{m+1}}V[G]_{\overline{W}_{m+1}}=N_{m+1}[G\times G_{m+1}]=N_{m}[G\times G_{m}]=V[G]_{\overline{W}_m}\end{equation}
   Indeed $G_m$ and $G_{m+1}$ differ by exactly one ordinal since they are given the sequences of sets of indiscernibles associated with essentially the same complete iterations (see also Proposition \ref{Proposition: indiscernible after ultrapower}).
\end{proof}
\begin{lemma}\label{Lemma: generic sums}
    Fix an internal iteration $(D_0,...,D_{n-1})$ of normal ultrafilters,  let \(d\) and \(d'\) be as in the paragraph following Theorem \ref{Them: sums of numral under UA}. Suppose that $d=d'$
and fix \(u : d\to \{0,1\}\) and \(x : d\to \omega\). Let $\overline W=\sum(D_0,...,D_{n-1})^{u}_x$. Then in $V[G]$ there is an internal iteration $(F_0,...,F_{\ell-1})$ of normal ultrafilters such that $V[G]_{\overline{W}}=V[G]_{F_0,...,F_{\ell-1}}$.
\end{lemma}
\begin{proof}
    An easy induction using the definition of $\sum(D_0,...,D_{n-1})^{u}_x$ in the case $d=d'$. Note that in the case that $m\notin d$, $D^{u,x}_m$ is an internal normal ultrafilter of $V[G]_{\overline{W}_m}$. If $m\in d$ then $m\in d'$ by our assumption and $V[G]_{\overline{W}_m}=V[G]_{\overline{W}_{m+1}}$ which follows by Equation \ref{N_m=N_{m+1}} if $u(m)=0$ or since $D^{u,x}_m$ is principal in the case that $u(m)=1$. 
\end{proof}
\subsection{Classifying the extensions of sums of normals}\label{subsection: classification}
In this section, we classify the extensions to \(V[G]\) of sums of normal ultrafilters; i.e., ultrafilters of the form $\Sigma(D_0,...,D_n)$. As expected, the proof is by induction on $n$. Recall that given a finite iteration $(D_0,...,D_n)$, we define $d=d(D_0,...,D_n)$ as the set of all $m\leq n$ such that $\delta_m\in j_{0,m}(\Delta)$ and $d'=d'(D_0,...,D_n)$ as the set of $m\in d$ such that $D_m=j_{0,m}(\vec{U})_{\delta_m}$.
\begin{theorem}\label{Thm: Classification}
     Let \(\overline{W}\) be a countably complete \(V[G]\)-ultrafilter extending \(\Sigma(D_0,...,D_n)\). Then \(\overline{W} = \Sigma(D_0,...,D_n)^u_x\) for some \(u:d'\rightarrow \{0,1\}\) and \(x:d\rightarrow \omega\).
\end{theorem}
\begin{proof}
We follow a similar argument to the one in 
Theorem \ref{Thm: Classification simple case}. Suppose inductively that we have classified the extensions of $W_m=\Sigma(D_0,...,D_{m-1})$ and let us classify the extensions of $W_{m+1}=\Sigma(D_0,...,D_{m})$. Let $\overline{W}_{m+1}$ be an ultrafilter on $\kappa^{m+1}$ extending of $W_{m+1}$ to a $V[G]$-ultrafilter. Note that \([\text{id}]_{\overline W_{m+1}} = (\overline\delta_0,...\overline \delta_m)\) for some increasing sequence of ordinals. Therefore there is a factor map $k:V_{W_{m+1}}\rightarrow j_{\overline{W}_{m+1}}(V)$ such that $j_{\overline{W}_{m+1}}\restriction V=k\circ j_{W_{m+1}}$ and $k([id]_{W_{m+1}})=[id]_{\overline{W}_{m+1}}$.  
  
Since \(\overline{W}_{m+1}\) extends \(W_{m+1}\),
\(\pi_*(\overline W_{m+1})\) extends \(W_m\) where \(\pi : \kappa^{m+1}\to \kappa^m\) denotes the projection to the first $m$-coordinates. By the induction hypothesis, it follows that \(\pi_*(\overline W_{m+1})\) must be equal to \(\overline{W}_m=\Sigma(D_0,...,D_{m-1})^{u'}_{x'}\) for some $u',x'$.
Therefore there is an elementary embedding \(\ell : V[G]_{\overline{W}_m} \to V[G]_{\overline W_{m+1}}\) such that \(\ell\circ j_{\overline{W}_m} = j_{\overline W_{m+1}}\) and \(\ell([id]_{\overline{W}_m}) =(\overline\delta_{0},...,\overline\delta_{m-1})\). 

Let \(\eta\) be the least ordinal such that \(\ell(\eta) > \overline \delta_m\) and let \(\overline{U}\) denote the \(V[G]_{\overline{W}_m}\)-ultrafilter on \(\eta\) derived from \(\ell\) using \(\overline{\delta}_m\). By Corollary \ref{cor: hulls and ultrapowers}, $\ell=j_{\overline{U}}$ and $[id]_{\overline{U}}=\overline{\delta}_m$. 
Let $e_m, i^{P_m},N_m$ be defined as before (after Theorem \ref{Them: sums of numral under UA}) for $\overline{W}_m=\Sigma(D_0,..,D_{m-1})^{u'}_{x'}$, and denote by $i_\alpha=i_{0,\alpha}^{P_m}\circ e_m:V_{W_n}\rightarrow N_m$ and let $i=i_{0,\infty}=i^{P_m}\circ e_m$. Let \(\tilde{U} = \overline{U}\cap N_m\), then there is a factor map $k_{\tilde{U}}:(N_{m})_{\tilde{U}}\rightarrow j_{\overline{W}_{m+1}}(V)$ such that $j_{\overline{U}}\restriction N_m=k_{\tilde{U}}\circ j_{\tilde{U}}$ and $k_{\tilde{U}}([id]_{\tilde{U}})=\overline{\delta}_m$. Then the following diagram commutes:
\begin{figure}[H]
    \centering
    
% https://q.uiver.app/#q=WzAsOSxbMCwyLCJWIl0sWzIsMiwiVl97V19tfSJdLFszLDEsIk5fbSJdLFs0LDEsIihOX20pX3tcXGJhcntVfX0iXSxbMCwwLCJWW0ddIl0sWzMsMCwiVltHXV97XFxiYXJ7V31fbX0iXSxbNSwxLCJqX3tcXGJhcntXfV97bSsxfX0oVikiXSxbNCwyLCJWX3tXX3ttKzF9fSJdLFs1LDAsIlZbR11fe1xcYmFye1d9X3ttKzF9fSJdLFswLDEsImpfe1dfbX0iXSxbMSwyLCJpIl0sWzQsNSwial97XFxiYXJ7V31fbX0iXSxbNyw2LCJrX3tEX219Il0sWzEsNywial97RF9tfSJdLFsyLDMsImpfe1xcdGlsZGV7VX19Il0sWzUsOCwial97XFxiYXJ7VX19Il0sWzMsNiwia197XFx0aWxkZXtVfX0iXSxbNiw4XSxbMiw1XSxbMCw0XV0=
\[\begin{tikzcd}
	{V[G]} &&& {V[G]_{\overline{W}_m}} && {V[G]_{\overline{W}_{m+1}}} \\
	&&& {N_m} & {(N_m)_{\overline{U}}} & {j_{\overline{W}_{m+1}}(V)} \\
	V && {V_{W_m}} && {V_{W_{m+1}}}
	\arrow["{j_{\overline{W}_m}}", from=1-1, to=1-4]
	\arrow["{j_{\overline{U}}}", from=1-4, to=1-6]
	\arrow[phantom, sloped,"\subseteq",from=2-4, to=1-4]
	\arrow["{j_{\tilde{U}}}", from=2-4, to=2-5]
	\arrow["{k_{\tilde{U}}}", from=2-5, to=2-6]
	\arrow[phantom, sloped,"\subseteq", from=2-6, to=1-6]
	\arrow[phantom, sloped,"\subseteq", from=3-1, to=1-1]
	\arrow["{j_{W_m}}", from=3-1, to=3-3]
	\arrow["i", from=3-3, to=2-4]
	\arrow["{j_{D_m}}", from=3-3, to=3-5]
	\arrow["{k}", from=3-5, to=2-6]
\end{tikzcd}\]
    \caption{The decomposition of $j_{\overline{W}_{m+1}}$.}
    \label{fig: Induction step characterization diagram}
\end{figure}

We claim that $i[D_m]\subseteq\tilde{U}$. To see this, let $X\in D_m$, then \begin{align*}
    [id]_{D_m}\in j_{D_m}(X)&\Rightarrow  \overline{\delta}_m\in k(j_{D_m}(X))\\ &\Rightarrow k_{\tilde{U}}([id]_{\tilde{U}})\in k_{\tilde{U}}(j_{\overline{U}}(i(X)))\\
    &\Rightarrow [id]_{\tilde{U}}\in j_{\overline{U}}(i(X)) \Rightarrow i(X)\in \tilde{U}
\end{align*}
Next we will prove that 
$\tilde{U}$ is equal to one of the following ultrafilters:
\begin{itemize}
    \item $i(D_m)$. 
    \item $p_{\delta^{\alpha+n}_{m}}$ for some $n<\omega$. 
    \item $D^{\alpha+n}_{m}$ for some $n<\omega$. 
\end{itemize} 
Once we prove the above, it will follow that $\tilde{U}$ is the ultrafilter that was denoted by $\tilde{D}_m$ which by the previous part generates the ultrafilter $D^{u,x}_m$ in $V[G]_{\overline{W}_m}$. Hence $\overline{U}=D^{u,x}_m$ for an appropriate extension of $u',x'$ to $u,x$ determined by the value of $\tilde{U}$. This will end the proof as by definition, $$\overline{W}_{m+1}=\Sigma(\overline{W}_m,\overline{U})=\Sigma(\overline{W}_m,D^{u,x}_m)=\Sigma(D_0,...,D_m)^{u}_x$$

 Let $\alpha$ be the first stage of the iteration such that the critical point of $i^{P_m}_{\alpha,\infty}$ is at least $i_{\alpha}(\delta_m)$. We will show that $i^{P_m}_{\alpha,\infty}[i_{\alpha}(D_m)]\subseteq \tilde{U}$. Note that every generator of $i_\alpha$ is less than $i_{\alpha}(\delta_m)$; this follows from our choice of $\alpha$ and the fact that $e_m$ is an iteration of ultrafilters on cardinals below $\delta_m$. Similarly, $i_\alpha$ is continuous at $\delta_m$ and therefore $i_\alpha[D_m]$ generates $i_\alpha[D_m]\cup F$, where $F$ is the tail filter on $i_\alpha(\delta_m)$. Applying Lemma \ref{lemma: normal generation} we conclude that $i_\alpha[D_m]$ generates $i_\alpha(D_m)$. Since $i[D_m]\subseteq \tilde{U}$ it follows that $i^{P_m}_{\alpha,\infty}[i_\alpha(D_m)]\subseteq \tilde{U}$.

We first consider the case where $m\notin d$, meaning $\delta_m\notin j_{0,m}(\Delta)$. Hence $\crit(i^{P_m}_{\alpha,\infty})>i_{\alpha}(\delta_m)$. It follows that $\tilde{U}=i_\alpha(D_m)=i(D_m)$. 

Now suppose that $m\in d$, namely $\delta_m\in j_{0,m}(\Delta)$. 
        The analysis of \(\tilde U\) is an application of Lemma \ref{lemma: tilde U} in the case \(M = N^m_{\alpha}\), $i=i^{P_m}_{\alpha,\infty}$ and the ultrafilter $i_\alpha(D_m)$ which is a normal ultrafilter on $i_\alpha(\delta_m)$, the minimal ordinal in the remaining part of the complete iteration $i^{P_m}_{\alpha,\infty}$ as computed in $N_{m,\alpha}$.
We conclude that either \(\tilde{U}\) has the desired form or else \(\tilde{U} = i_{\alpha+\omega}(D_m)\). But 
the latter cannot occur, because \(\tilde{U}\) extends to a countably complete \(V[G]_{\overline{W}_m}\)-ultrafilter (namely, \(\overline U\)), whereas \(i_{\alpha+\omega}(D_m)\) does not since $i_{\alpha+\omega}(\delta_m)$ has countable cofinality in $V[G]_{\overline{W}_{m}}$.$\qedhere_{\text{Theorem \ref{Thm: Classification}}}$
\end{proof}
\section{Applications}\label{Section:Applications}
Our first application resolves the problem of whether Weak UA is equivalent to UA (see \cite[Question 9.2.4]{GoldbergUA}).     \begin{lemma}\label{Lemma: Sufficient condition for Weak UA}
            Assume that the Mitchell order is linear on normal ultrafilters, and that for every $\sigma$-complete ultrafilter $U$ there is an internal iteration of normal ultrafilters $(D_0,...,D_{n-1})$ such that $V_U=V_{D_0,...,D_{n-1}}$. Then Weak UA holds.
        \end{lemma}
        \begin{proof}
            Granting the linearity of the Mitchell order on normal measures, \cite[Prop. 2.3.13]{GoldbergUA} states that the Ultrapower Axiom for normal ultrafilters holds. It is then not hard to show that the Ultrapower Axiom holds for internal (finite) iterations of normal ultrafilters (see for example the proof \cite[Prop. 8.3.43]{GoldbergUA}).  By the second assumption of the lemma, this suffices to compare (without commutativity) every two ultrapowers of $\sigma$-complete ultrafilters.
        \end{proof}
        \begin{lemma}
            Suppose that the Mitchell order is linear in $V$ and that each normal ultrafilter of $V[G]$ is generated by an ultrafilter of $V$. Then the Mitchell order is linear in $V[G]$.
        \end{lemma}
        \begin{proof}
            Let $U,W\in V[G]$ be distinct normal measures on a cardinal $\kappa$. Let $U_0=U\cap V$ and $W_0=W\cap V$. Then $U_0,W_0\in V$ generate $U,W$ respectively and therefore they are distinct normal measures in $V$. Suppose without loss of generality that $U_0\triangleleft W_0$, and let $k:M^V_{W_0}\rightarrow j_W(V)$ be the factor map. The critical point of $k$ is greater than $\kappa$ (since both are $W_0$ and $W$ are normal). It follows that $U_0\subseteq k(U_0)\in j_W(V)\subseteq M_W$. Since $V[G]$ and $M_W$ have the same subsets of $\kappa$,  $k(U_0)$ also must generate $U\in M_W$. 
        \end{proof}
    \begin{corollary}[UA]
         Assume GCH  and that $\kappa$ is the least measurable limit of measurables. Let $\vec{U}$ be the sequence of normal measures on all the measurables below $\kappa$. Let $G$ be $V$-generic for $\mathbb{P}_\kappa$. Then $V[G]\models \text{weak UA}+\neg\text{UA}$.
    \end{corollary}
    \begin{proof}
       By \cite[2.3.1]{GoldbergUA}, in $V$ the Mitchell order is linear and since every normal ultrafilter in $V[G]$ is generated by a ground model normal measure, the previous lemma can be used to infer that the Mitchell order is linear in $V[G]$. Therefore by  Lemma \ref{Lemma: Sufficient condition for Weak UA}, it remains to prove that every ultrapower by a $\sigma$-complete ultrafilter in $V[G]$ is equal to an ultrapower by a sum of normals. Let $U\in V[G]$ be a $\sigma$-complete ultrafilter. Since $\kappa$ is the least measurable in $V[G]$, $U$ is $\kappa$-complete there. By \cite{GITIK_KAPLAN_2023}, $U_0=U\cap V$ is a $\kappa$-complete ultrafilter on $\kappa$ in $V$, and by Theorem \ref{Them: sums of numral under UA}, 
       Lemma \ref{Lemma: generic sums}, and Theorem \ref{Thm: Classification}, we obtain $V[G]_U=V[G]_W$ where $W$ is the sum of a finite iteration of normal ultrafilters.

       To see that UA fails in $V[G]$, note that 
       in $V[G]$, by Theorem \ref{theorem: non-rigid ultrapower}, there is an ultrapower $M$ of the universe that admits a nontrivial elementary embedding $k: M \to M$. This would be impossible if UA held in $V[G]$.
       The reason is that assuming UA, by \cite[Thm. 5.2]{GoldbergUniqueness} there is at most one elementary embedding $j : V[G]\to M$.
       Therefore $k\circ j = j$,. and so by a standard lemma on the Rudin--Keisler order (proved for example in \cite[Cor. 4.29]{GoldbergUniqueness}), $k$ would be the identity.
    \end{proof}
We note that the GCH assumption can be replaced by an argument using UA to show that no new $V$-measures appear in the extension. 

Our final application is to a natural question. Can two countably complete uniform ultrafilters on distinct cardinals have the same ultrapower? That is, given such ultrafilters $U_0,U_1$, can $V_{U_0}$ be equal to $V_{U_1}$? One obstruction is the following:
\begin{proposition}
    Suppose that \(U_0\) and \(U_1\) are countably complete uniform ultrafilters on regular cardinals \(\kappa_0\) and \(\kappa_1\), and assume \(V_{U_0} = V_{U_1}\). Then \(\kappa_0 = \kappa_1\).
    \begin{proof}
        Since \(j_{U_0}\) and \(j_{U_1}\) are elementary embeddings from \(V\) into the same inner model, we can appeal to a theorem of Woodin \cite[Theorem 3.4]{GoldbergUniqueness} to obtain that \(j_{U_0}\restriction \text{Ord} = j_{U_1}\restriction \text{Ord}\). Assume without loss of generality that \(\kappa_0 \leq \kappa_1\). Then
        since \(U_1\) is uniform on \(\kappa_1\), \(j_{U_1}\) is discontinuous at \(\kappa_1\). It follows that \(j_{U_0}\) is discontinuous at \(\kappa_1\). Since \(\kappa_1\) is regular and \(j_{U_0}\) is discontinuous at \(\kappa_1\),
        we cannot have \(\kappa_0 < \kappa_1\) by
        \cite[Lemma 2.2.34]{GoldbergUA}. Therefore \(\kappa_0 = \kappa_1\).
    \end{proof}
\end{proposition}
The following example shows that the assumption above that \(\kappa_0\) and \(\kappa_1\) are regular cardinals is necessary.
It also demonstrates one of the complications arising in the attempt to extend our results on extensions of \(\kappa\)-complete ultrafilters on \(\kappa\) to arbitrary countably complete ultrafilters.

For the remainder of the paper, let $\lambda$ be a measurable cardinal
and let $\kappa>\lambda$ denote the least limit of measurable cardinals of cofinality $\lambda$.  Let \(\Delta\) denote the set of measurable cardinals strictly between \(\lambda\) and \(\kappa\), and let \(\vec U : \Delta\to V\) assign to each such measurable cardinal a normal ultrafilter. Finally, let \(\mathbb P_\kappa\) be the descrete Magidor product of Prikry forcings associated with \(\vec U\) and let \(G\subseteq \mathbb P_\kappa\) be \(V\)-generic.

\begin{lemma}\label{lemma: forcing above lambda}
    If \(D\) is a normal ultrafilter on \(\lambda\),
    then \(j_D^{V[G]}\restriction V = i\circ j_D\)
    where \(i : V_D\to N\) is the complete iteration of \(V_D\) by \(j_D^V(\vec U)\restriction (\kappa,j_D^V(\kappa))\). Moreover \(j_D^{V[G]}(G) = (G\cap V_D)\times G_{\vec s}\)  where \(\vec s\) is the sequence of sets of indiscernibles associated with the complete iteration of \(j_D^V(\vec U)\restriction (\kappa,j_D^V(\kappa))\).
    \begin{proof}
        Note that \(G\cap V_D\) is \(V_D\)-generic
        on \(\mathbb P_\kappa(j_D(\vec U)\restriction \kappa)\), and therefore by the Mathias criterion and Lemma \ref{Lemma: Generic}, \[H = (G\cap V_D)\times G_{\vec s}\] is \(N\)-generic on \(i\circ j_D(\mathbb P_\kappa)\). Moreover, if \(p\in G\), then \(j_D^V(p)\restriction \kappa\in G\cap V_D\) and \(j_D^V(p)\restriction (\kappa,j_D(\kappa))\) is a pure condition, and so \(i(j_D^V(p))\in (G\cap V_D)\times G_{\vec s}\). It follows that
        \(i\circ j_D : V \to N\) lifts to an elementary embedding \(j : V[G]\to N[H]\) with \(j(G) = H\).

        To show that \(j = j_D^{V[G]}\) it suffices by 
        Corollary \ref{cor: hulls and ultrapowers} to show that \(N[H] = \hull^{N[H]}(j[V[G]]\cup \{\lambda\})\). Since \(i[V_D] = \hull^N(j[V]\cup \{\lambda\})\), we have
        \[i[V_D]\cup \{G_{\vec s}\}\subseteq \hull^{N[H]}(j[V[G]]\cup \{\lambda\})\] and \(N[G_{\vec s}]\subseteq \hull^{N[H]}(j[V[G]]\cup \{\lambda\})\) by Lemma \ref{Lemma: Generic} applied in $V_D$.
        Since \(N[H] = N[G_{\vec s}][G\cap V_D]\)
        and \(N[G_{\vec s}]\cup \{G\cap V_D\}\subseteq \hull^{N[H]}(j[V[G]]\cup \{\lambda\})\), it follows that 
        \(N[H] = \hull^{N[H]}(j[V[G]]\cup \{\lambda\})\), as desired.
    \end{proof}
\end{lemma}

    \begin{proposition}
        In \(V[G]\), there are countably complete uniform ultrafilters on \(\lambda\) and \(\kappa\) with the same ultrapower.
    \end{proposition}
    \begin{proof}
        Let \(D\) be a normal ultrafilter on
        \(\lambda\). Let \(f : \lambda\to V_\kappa\) be the increasing enumeration of \(\vec U\), and let \(U = [f]_D\). Then in \(V_D\), \(U\) is a normal ultrafilter on the least measurable \(\gamma > \kappa\). 
        Finally, let \(W = \Sigma(D,U)\). Then \(W\) is a uniform \(\lambda\)-complete ultrafilter on \(\kappa^2\).
        
        We claim that there is an extension \(W^*\) of \(W\) to a uniform \(\lambda\)-complete \(V[G]\)-ultrafilter on \(\kappa^2\) such that
        \(V[G]_{W^*} = V[G]_D\). 

        We will define \(W^*\) as the analog of the canonical lift of \(W\) (defined in Section \ref{Section: Canonical extension}) to this situation. The restricted ultrapower of \(W^*\) will be the elementary embedding \(i\circ j_W^V \) where \(i : V_{W}\to N\) is the complete iteration of \(V_W\) by \(j_W(\vec U)\restriction (\kappa,j_W(\kappa))\). The image generic will be
        \((G\cap V_W) \times G_{\vec t}\) where \(\vec t\) is the sequence of sets of indiscernibles associated with the complete iteration of \(j_W(\vec U)\). The seed will just be
        \(i([id]_W)\). These three ingredients uniquely determine \(W^*\), and it is not hard to verify that 
        a \(V[G]\)-ultrafilter \(W^*\supseteq W\) with these invariants exists.

        Note that the complete iteration of \(V_W\) by \(j_W(\vec U)\restriction (\kappa,j_W(\kappa))\) is just the tail of the complete iteration of \(V_D\) by \(j_D(\vec U)\restriction (\kappa,j_D(\kappa))\) after applying the first measure. 
        Therefore by Lemma \ref{lemma: forcing above lambda},
        \(V[G]_{W^*} = N[(G\cap N)\times G_{\vec t}] = V[G]_D\), noting that the sequence \(\vec t\) differs from the sequence \(\vec s\) of sets of indiscernibles associated with \(j_D(\vec U)\restriction (\kappa,j_D(\kappa))\) by just one ordinal.

        It remains to show that \(W^*\) is a uniform ultrafilter on \(\kappa^2\). The reason is that \(j_{W^*}\restriction V[G]_\kappa = j_D^{V[G]}\restriction V[G]_\kappa\), while \(j_{W^*} \neq j_D^{V[G]}\), the latter following from the fact that \[j^{V[G]}_D(G) = G\times G_{\vec s}\neq G\times G_{\vec t}= j_{W^*}(G)\] If \(W^*\) were not uniform on \(\kappa^2\), then \(W^*\) would be Rudin--Keisler equivalent to an ultrafilter \(Z\) on some \(\gamma < \kappa\) derived from \(j_{W^*}\) and some \(\xi < j_{W^*}(\gamma)\). But then \(Z\) is also derived from \(j_D\) and \(\xi\), so \(W^* \leq_{RK} Z \leq_{RK} D\). Since \(D\) is normal and \(W^*\) is nonprincipal, it follows that \(D\) and \(W^*\) are Rudin--Keisler equivalent, contrary to the fact that \(j_{W^*} \neq j_D^{V[G]}\).
    \end{proof}
\section{Problems}\label{Section: Problems}
We list out a few related problems we did not address:
\begin{question}
    Can we characterize all the $\sigma$-complete extensions of a $\sigma$-complete ultrafilter on $V$ after the discrete Magidor product? In particular, are there only countably many extensions? 
\end{question}
\begin{question}
    Can we characterize the $\sigma$-complete extensions of sums of normals after other types of iterations of Prikry forcing?
\end{question}
\begin{question}
    Working over any ground model $V$, can we find a characterization of all the extensions of a countably complete ultrafilter to a countably complete ultrafilter after the discrete Magidor product?
\end{question}
We conjecture that if $\kappa$ is strongly compact then after a discrete Magidor product below $\kappa$, there is a $\kappa$-complete $V$-ultrafilter over $\kappa$ which has uncountably many lifts.
\begin{question}
Is there a forcing that preserves UA and adds a subset $X$ to the least supercompact cardinal $\kappa$ such that $X\notin V[Y]$ for any $Y\subseteq V$ of cardinality less than $\kappa$?
\end{question}
\subsection{Acknowledgments}
The authors would like to thank Eyal Kaplan for many discussions on the subject and Omer Ben-Neria for comments on our first draft. Finally, we would like to express our gratitude to the referee of this paper for many valuable comments and corrections.
\bibliographystyle{amsplain}
\bibliography{ref}
\end{document}